\documentclass[reqno]{amsart}
\usepackage[english]{babel}
\usepackage{amscd,amssymb,amsmath,amsfonts,latexsym,amsthm}
\usepackage{ulem}
\usepackage{inputenc}
\usepackage{graphicx,color}
\usepackage{longtable}
\usepackage{multirow}
\usepackage{amsmath}
\usepackage{graphicx}
\usepackage{float}
\newtheorem{thm}{Theorem}[section]
\newtheorem{cor}[thm]{Corollary}
\newtheorem{lem}[thm]{Lemma}
\newtheorem{prop}[thm]{Proposition}
\newtheorem{defn}[thm]{Definition}

\numberwithin{equation}{section}

\begin{document}

\title[Abelian extensions of five-dimensional solvable Leibniz algebras]{Abelian extensions of five-dimensional solvable Leibniz algebras}

\author[Khudoyberdiyev A.Kh.]{Abror Khudoyberdiyev$^{1,2}$}

\address{$^1$V.I.Romanovskiy Institute of Mathematics Uzbekistan Academy of Sciences. 9,  University street, 100174, Tashkent, Uzbekistan}

\address{$^2$National University of Uzbekistan named after Mirzo Ulugbek. 4, University street, 100174, Tashkent, Uzbekistan}

\author[Sheraliyeva S.A.]{Surayyo Sheraliyeva$^1$}

\email{{\tt khabror@mail.ru, abdiqodirovna@mail.ru}}

\begin{abstract}
In this work, we extend the central extension method for solvable Leibniz algebras. Using this method, a complete classification of one-dimensional abelian extensions of five-dimensional solvable Leibniz algebras with a non-trivial three-dimensional nilradical is obtained.
Furthermore, we explore extensions of solvable Leibniz algebras whose nilradical is null-filiform, establishing that, in this case, there exists a unique solvable abelian extension.
\end{abstract}

\maketitle



\section{Introduction}
Leibniz algebras were introduced by Bloch in 1965 \cite{Bloch}, who referred to them as $D$-algebras. They were later studied by Loday and Cuvier \cite{Cuv, Loday} as a non-antisymmetric analog of Lie algebras.
Since then, many researchers have worked on Leibniz algebras, and as a result, they have found analogues of important theorems in Lie theory, such as the analogue of Levi's theorem, which was proven by Barnes \cite{Barnes}. He showed that any finite-dimensional complex Leibniz algebra decomposes into a semidirect sum of its solvable radical and a semisimple Lie algebra. Thus, the primary challenge in the classification of finite-dimensional complex Leibniz algebras lies in studying the solvable part.
To classify solvable Leibniz algebras, it is necessary to consider nilpotent Leibniz algebras, which serve as the nilradicals, just as in the case of Lie algebras.

Since Loday has reintroduced the concept, numerous researchers have contributed to the classification of various classes of finite-dimensional Leibniz algebras \cite{Omirov, AyupovBook, Casas}.
A complete classification of all two-dimensional algebras is provided in the works of Loday, whereas the first constructions of three-dimensional Leibniz algebras were presented by Ayupov and Omirov \cite{Ayupov}. A complete list of three-dimensional Leibniz algebras is presented in the works of  Casas et al. \cite{Casas}. A comprehensive classification of four-dimensional nilpotent Leibniz algebras is provided in the following works \cite{Albeverio, Kaygorodov}. The full classification of four-dimensional, non-nilpotent solvable Leibniz algebras is established in \cite{Canete}, while the five-dimensional solvable Leibniz algebras with a three-dimensional nilradical are explicitly determined in \cite{Khudoyberdiyev}.

The extension method is widely regarded as one of the most effective approaches for classifying algebras. The central extension method enables the classification of all finite-dimensional nilpotent algebras. However, since central extensions produce only algebras with a nonzero center, not all solvable algebras can be obtained through this method. Therefore, Sund, in \cite{Sund2}, generalized the central extension method for solvable Lie algebras, demonstrating that all solvable algebras with non-split nilradicals can be obtained through this approach. By applying the method developed by Sund, it was determined that the extensions of solvable Lie algebras with a filiform nilradical \cite{Sheraliyeva1}.

The study of abelian extensions in the context of Leibniz algebras was initially introduced in \cite{Casas1}. Central extensions of Leibniz algebras were examined in \cite{Rakhimov}, where a complete classification of central extensions for null-filiform Leibniz algebras was provided. Additionally, the notion of non-abelian extensions was introduced in \cite{Liu}. It is important to note that the extensions of solvable Leibniz algebras and superalgebras have been examined in \cite{Camacho} and \cite{Sheraliyeva}, with a focus on the central extensions of solvable algebras in these works. As previously mentioned, since central extensions do not generate all solvable algebras, it is necessary to adapt the method developed by Sund for solvable Lie algebras to the context of Leibniz algebras. In this work, we generalize this method for Leibniz algebras. To achieve this, we employ the notion of non-abelian extensions introduced in \cite{Liu}.

In Section 2, we demonstrate that solvable Leibniz algebras with a non-split nilradical can be characterized as extensions of lower-dimensional solvable Leibniz algebras.
In Section 3, we investigate extensions of solvable Leibniz algebras with a null-filiform nilradical and establish that, in this case, there exists a unique solvable abelian extension.
In Section 4, we provide a complete classification of one-dimensional abelian extensions of five-dimensional solvable Leibniz algebras with a non-trivial three-dimensional nilradical.

\section{ Preliminaries}

In this section, we recall some fundamental notions and concepts that will be used throughout the paper.
All the algebras below will be over the complex field, and all the linear maps will be $\mathbb C$-linear unless otherwise stated. The notation $\langle S \rangle$ means the $\mathbb C$-subspace generated by $S$.

\begin{defn}
A vector space with bilinear bracket $(\mathfrak{g},[\cdot,\cdot]) $ is called a (right) Leibniz algebra if for any  $x,y,z\in \mathfrak{g}$ to the so-called Leibniz identity
 \begin{equation*}{\label{l1}} [[x,y],z] = [x,[y,z]]+ [[x,z],y],\end{equation*}
holds.\end{defn}

For a given Leibniz algebra $(\mathfrak{g},[\cdot,\cdot])$, the sequences of
two-sided ideals are defined recursively as follows:
$$\mathfrak{g}^1=\mathfrak{g}, \ \mathfrak{g}^{k+1}=[\mathfrak{g}^k,\mathfrak{g}],  \ k \geq 1, \qquad \qquad
\mathfrak{g}^{[1]}=\mathfrak{g}, \ \mathfrak{g}^{[s+1]}=[\mathfrak{g}^{[s]},\mathfrak{g}^{[s]}], \ s \geq 1.$$

These are said to be the lower central and the derived series of $\mathfrak{g}$, respectively.

\begin{defn} A Leibniz algebra $\mathfrak{g}$ is said to be
nilpotent (respectively, solvable), if there exists $n\in\mathbb
N$ $(m\in\mathbb N),$ such that $\mathfrak{g}^{n}=\{0\}$ $(\mathfrak{g}^{[m]}=\{0\})$.\end{defn}

Note that the sum of two nilpotent ideals of the Leibniz algebra is also nilpotent. Therefore, the maximal nilpotent ideal always exists.
The maximal nilpotent ideal of a Leibniz algebra is said to be the nilradical of the algebra.

\begin{defn}
A Leibniz algebra $\mathfrak{g}$ is called null-filiform if  $\operatorname{dim}\mathfrak{g}^{i}=n+1-i$ for $1\leq i\leq n+1$,
where $n=\operatorname{dim}\mathfrak{g}.$\end{defn}

It is known that an arbitary $n$-dimensional null-filiform Leibniz algebra is isomorphic to the algebra (see \cite{Omirov})
  $$NF_n: [e_i,e_1]=e_{i+1}, \quad 1 \leq i \leq n-1.$$

Solvable Leibniz algebras with nilradical $NF_n$ are classified in \cite{Karimjonov}, and it was shown that up to isomorphism there exists a unique solvable Leibniz algebra
 $$R=\begin{cases}
[e_i,e_1]=e_{i+1}, & 1\leq i\leq n-1,\\
[e_i,x]=-ie_i, & 1\leq i \leq n,\\
[x,e_1]=e_1.\end{cases}$$

\begin{defn} A representation of a Leibniz algebra $(\mathfrak{g},[\cdot,\cdot])$ is a triple $(V,l,r),$ where $V$ is a vector space equipped with two linear maps $l:\mathfrak{g}\rightarrow \mathfrak{gl}(V)$ and $r:\mathfrak{g}\rightarrow \mathfrak{gl}(V),$ such that the following equalities hold:
 \begin{equation}{\label {c1}}r_{[x,y]}=r_y \circ r_x - r_x \circ r_y,\quad l_{[x,y]}= r_y \circ l_x  - l_x \circ r_y,\end{equation}
 \begin{equation}{\label {c3}}l_x\circ l_y=-l_x\circ r_y, \quad \forall x,y\in \mathfrak{g}.\end{equation}
 \end{defn}

\begin{defn} The extension of an algebra $\mathfrak{g}$ by $V$ is a short exact sequence
$$0\longrightarrow V\mathop{\longrightarrow}\limits^{\varepsilon} \widehat{\mathfrak{g}} \mathop{\longrightarrow}\limits^{\lambda} \mathfrak{g} \longrightarrow0,$$
such that $Im \varepsilon= Ker \lambda$. The extension is called a central, if $ Ker \lambda = Z(\widehat{\mathfrak{g}}),$ i.e., $Z(\widehat{\mathfrak{g}}) = Im \varepsilon \simeq V.$
\end{defn}

In \cite{Liu}, J.Liu, Y.Sheng, and Q.Wang introduced the concept of a non-abelian extension for left Leibniz algebras and gave a necessary and sufficient condition for the given extension to be a Leibniz algebra. Since we are considering the right Leibniz algebras in this article, we present these results for the right Leibniz algebras.

Let $\mathfrak{g},$ $\mathfrak{h}$ be Leibniz algebras, and $l:\mathfrak{g}\rightarrow \mathfrak{gl}(\mathfrak{h}),$ $r:\mathfrak{g}\rightarrow \mathfrak{gl}(\mathfrak{h})$ and
  $\omega: \mathfrak{g} \otimes \mathfrak{g} \rightarrow \mathfrak{h}$ is a bilinear map.
 For the vector space $\mathfrak{g}\oplus \mathfrak{h}$, we define the multiplication
\begin{equation}{\label{fff}}[x+a, y+b]_{(l,r,\omega)}=[x,y]_g+\omega(x,y)+l_xb+r_ya + [a,b]_{\mathfrak{h}}, \quad x,y\in \mathfrak{g}, a, b \in \mathfrak{h}.\end{equation}

\begin{prop} With the above notations, $(\mathfrak{g}\oplus \mathfrak{h}, [\cdot,\cdot]_{l,r,\omega})$ is a Leibniz algebra if and only if $l,r,\omega$ satisfy the following are identities:
\begin{equation*}{\label {pr1}}r_x[a, b]=[r_x a, b]+[a, r_x b],\end{equation*}
\begin{equation*}{\label {pr2}}l_x[a, b]=[l_x a, b]-[l_xb, a],\end{equation*}
\begin{equation*}{\label {pr3}}[a, l_xb+r_xb]=0,\end{equation*}
\begin{equation*}{\label {pr4}}r_x \circ r_y - r_y \circ r_x=r_{[y,x]}+\operatorname{ad}_{\omega(x,y)}^R,\end{equation*}
\begin{equation*}{\label {pr5}}r_x \circ l_y - l_y \circ r_x = l_{[y,x]}+\operatorname{ad}_{\omega(x,y)}^L,\end{equation*}
\begin{equation*}{\label {pr6}}l_x(l_y+r_y)=0,\end{equation*}
\begin{equation*}{\label {pr7}}\omega([x,y],z)-\omega(x,[y,z])-\omega([x,z],y)=l_x\omega(y,z)+r_y\omega(x,z)-r_z\omega(x,y),\end{equation*}
where $\operatorname{ad}_{x}^R(y) = [y,x],$ $\operatorname{ad}_{x}^L(y) = [x,y].$

\end{prop}

The Leibniz algebra $(\mathfrak{g}\oplus \mathfrak{h}, [\cdot,\cdot]_{l,r,\omega})$ will be denoted by $\mathfrak{g}(\omega,l,r).$
In case of  $\mathfrak{h}$ is abelian, we get that $l:\mathfrak{g}\rightarrow \mathfrak{gl(h)}$ and $r:\mathfrak{g}\rightarrow \mathfrak{gl(h)}$ is a representation and $\omega: \mathfrak{g} \otimes \mathfrak{g} \rightarrow \mathfrak{h}$, a bilinear map satisfying
\begin{equation}{\label {c4}}\omega([x,y],z)-\omega(x,[y,z])-\omega([x,z],y)-l_x\omega(y,z)-r_y\omega(x,z)+r_z\omega(x,y)=0.\end{equation}

The bilinear maps $\omega$  satisfying \eqref{c4} we call a 2-cocycles on $\mathfrak{g}$ with respect to the pair $(l,r)$ and the set of all such 2-cocycles denote by $Z^2(\mathfrak{g},l,r).$

The 2-coboundary on $\mathfrak{g}$ with respect to the pair $(l,r)$ is defined as follows
$$df(x,y)=f([x,y])-l_{\varphi(x)}f(y)-r_{\varphi (y)}f(x), \quad x,y\in \mathfrak{g}, $$
where $\varphi\in \operatorname{Aut}(\mathfrak{g})$ and $f \in \operatorname{Hom}(\mathfrak{g},\mathfrak{h})$. The set of such bilinear maps is denoted by $B^2(\mathfrak{g},l,r).$

Let $\mathfrak{g}$ be a solvable Leibniz algebra, $N$ its nilradical, and $Z(N)$ the center of $N$.
Suppose that the Leibniz algebra $\widehat{\mathfrak{g}}=\mathfrak{g}(\omega,l,r)$ is the extension of the given algebra $\mathfrak{g}$ generated by $\omega, l, r$, and denote by $\widehat{N}$ a nilradical of the algebra $\widehat{\mathfrak{g}}$.
Let $\omega^0=\omega|_{N\times N}$ and $l^0=l |_N,$ $r^0=r|_N$. Denote $$G_{\omega^0}=\{X\in N:  \omega^0(N,X)=(0) \}.$$

\begin{lem}
The nilradical $\widehat{N}$ of $\widehat{\mathfrak{g}}$ is a central extension of $N$ by $\mathfrak{h}$  if and only if $N \subseteq \operatorname{ker} l \cap \operatorname{ker} r.$

 Moreover, if  $N \subseteq \operatorname{ker} l \cap \operatorname{ker} r$, then $Z(\widehat{N}) =\mathfrak{h}$ if and only if $G_{\omega^0}\cap Z(N)=0.$
\end{lem}

\begin{proof} Consider an extension $N(\omega^0,l^0, r^0)=(N\oplus \mathfrak{h}, [\cdot,\cdot]_{l_0,r_0,\omega_0}).$ Then the extension
$$0\rightarrow \mathfrak{h} \rightarrow N(\omega^0,l^0, r^0)\rightarrow N\rightarrow 0$$
is central if and only if $l^0=0, r^0=0.$ Then we have $N \subseteq \operatorname{ker} l \cap \operatorname{ker} r$ and $N(\omega^0,l^0,r^0)=N(\omega^0)\subset \widehat{N}$.
We show that $\widehat{N} \subset N(\omega^0)$. Since $\widehat{N}$ is an ideal of $\widehat{\mathfrak{g}}$, then it is an extension of subalgebra $M \subset \mathfrak{g}$ by $\mathfrak{h}$. It follows that $M$ is a nilpotent subalgebra of $\mathfrak{g}$ containing $N$. Since $N$ is a nilradical of $\mathfrak{g},$ we have  $M \subset N$ which follows $\widehat{N} \subset N(\omega^0).$

 Now, assume that $N \subseteq \operatorname{ker} l \cap \operatorname{ker} r,$ then we have $Z(\widehat{N})=(G_{\omega^0}\cap Z(N))\oplus \mathfrak{h}.$
Thus $Z(\widehat{N})=\mathfrak{h}$ if and only if $G_{\omega^0}\cap Z(N)=0$.
\end{proof}

Given two extensions $\widehat{\mathfrak{g}}_1=\mathfrak{g}(\omega^1,l^1,r^1)$ and
$\widehat{\mathfrak{g}}_2=\mathfrak{g}(\omega^2,l^2,r^2),$ such that $Z(\widehat{N}_1)=Z(\widehat{N}_2)=\mathfrak{h},$ where $\widehat{N}_1$ and $\widehat{N}_2$ are nilradicals of $\widehat{\mathfrak{g}}_1$ and $\widehat{\mathfrak{g}}_2.$

\begin{prop}
Two Leibniz algebras $\widehat{\mathfrak{g}}_1=\mathfrak{g}(\omega^1, l^1, r^1)$ and $\widehat{\mathfrak{g}}_2=\mathfrak{g}(\omega^2,l^2, r^2)$ are isomorphic if and only if there exists $\varphi\in \operatorname{Aut} (\mathfrak{g}),$ $\psi \in \operatorname{Aut} (\mathfrak{h})$ and $f \in \operatorname{Hom}(\mathfrak{g},\mathfrak{h}),$ such that
\begin{equation}{\label{l4}}\omega^2(\varphi(x),\varphi(y))=\psi \big(\omega^1(x,y)\big)+f([x,y])-l^2_{\varphi(x)}f(y)-r^2_{\varphi(y)}f (x),\end{equation}
\begin{equation}{\label{l5}}l^2_{\varphi(x)}\psi(a)=\psi\big(l^1_x a \big),\end{equation}
\begin{equation}{\label{l6}}r^2_{\varphi(y)}\psi(a)=\psi\big(r^1_y a \big),\end{equation}
for any $x, y \in \mathfrak{g},$ $a\in \mathfrak{h}.$
\end{prop}

\begin{proof}
Let $\Psi:\widehat{\mathfrak{g}}_1\rightarrow \widehat{\mathfrak{g}}_2$ be an isomorphism.
Choose bases $\{e_1, e_2, \dots, e_n\}$ and $\{v_1, v_2, \dots, v_k\}$ for the vector spaces $\mathfrak{g}$ and $\mathfrak{h},$ respectively. Since $\mathfrak{h}$ is an ideal of $\widehat{\mathfrak{g}},$ and $Z(\widehat{N}_1)=Z(\widehat{N}_1)=\mathfrak{h}$, it follows that $\mathfrak{h}$ is invariant under the isomorphism $\Psi$. Thus, we can realize $\Psi$ as a matrix relative  to a suitable basis for $\mathfrak{g}\oplus \mathfrak{h}$ which is assumed to contain a basis for $\mathfrak{g}$ and a basis for $\mathfrak{h}$:
$$\Psi= \begin{pmatrix}
\varphi & 0\\
f & \psi \end{pmatrix},  \quad  \varphi \in \operatorname{End}\mathfrak{g}, \ \psi \in \operatorname{Aut} \mathfrak{h}, \ f\in \operatorname{Hom}(\mathfrak{g},\mathfrak{h}).$$

Since $\Psi$ preserves the Leibniz products, we have
\begin{equation}{\label{l7}} \Psi\big([x+a,y+b]_1\big)=
\varphi([x,y])+f([x,y])+\psi\big(\omega^1(x,y)+ l_x^1b+r_y^1a\big).\end{equation}

On the other hand,
\begin{equation}{\label{l8}}\big[\Psi(x+a),\Psi(y+b)\big]_2=\big[\varphi (x), \varphi (y)\big]+\omega^2\big(\varphi (x),\varphi (y)\big)+l_{\varphi(x)}^2\big(f(y)+\psi(b)\big)+r_{\varphi(y)}^2\big(f(x)+\psi(a)\big).\end{equation}

Then, we get that $\varphi \in \operatorname{Aut} (\mathfrak{g}).$ Letting $a=b=0$ and comparing \eqref{l7} and \eqref{l8}, we have
\begin{equation*}{\label{l9}}\omega^2(\varphi(x),\varphi (y))=\psi \big(\omega^1(x,y)\big)+f([x,y])-l^2_{\varphi (x)}f(y)-r^2_{\varphi (y)}f(x).\end{equation*}

Taking this equality into account and comparing equations \eqref{l7} and \eqref{l8}, we have:
$$\psi\big(l^1_x b \big)-l^2_{\varphi(x)} \psi(b)=r^2_{\varphi(y)} \psi (a)-\psi\big(r^1_y a \big).$$

Letting $a=0,$ we obtain
$$l^2_{\varphi(x)}\psi(b)=\psi\big(l^1_x b \big),$$
and $b=0,$ theregives
$$r^2_{\varphi(y)}\psi(a)=\psi\big(r^1_y a \big).$$

 Conversely, if \eqref{l4}, \eqref{l5} and \eqref{l6} hold, then it is readily verified that the Leibniz algebras $\mathfrak{g}(\omega^1,l^1,r^1)$ and $\mathfrak{g}(\omega^2,l^2,r^2)$ are isomorphic.
\end{proof}

If we put $\omega\circ \varphi (x,y) = \omega (\varphi (x), \varphi (y))$ and $f \circ \omega (x,y) =f (\omega (x, y)),$ then the equation \eqref{l4} can be rewritten as
\begin{equation}\label{l98}\omega^2\circ \varphi=\psi \circ \omega^1+df,\end{equation}
where $df(x,y) =f([x,y])-l^2_{\varphi (x)}f(y)-r^2_{\varphi (y)}f(x)$, i.e., $df \in B^2(\mathfrak{g},l^2,r^2)$.

Therefore, $\omega^2=\psi \circ \omega^1\circ\varphi^{-1}+(df)\circ\varphi^{-1}$.
Hence, we have
$$(df)\circ\varphi^{-1}(x,y)=f [\varphi^{-1}(x),\varphi^{-1}(y)]-l_{\varphi(x)}f(\varphi^{-1}(y))-r_{\varphi(y)}f(\varphi^{-1}(x))=$$
$$=f\circ\varphi^{-1}[x,y]-l_{\varphi(x)}\big(f\circ\varphi^{-1}(y)\big)-r_{\varphi(y)}\big(f\circ\varphi^{-1}(x)\big).$$
Hence $(df)\varphi^{-1}=d(f\circ\varphi^{-1})\in B^2(\mathfrak{g},l^2,r^2).$

Furthermore,  from equations \eqref{l5} and \eqref{l6}, we get
\begin{equation}\label{l99}\begin{cases}l^2_{x}=\psi\circ l^1_{\varphi^{-1}(x)}\circ\psi^{-1},\\
r^2_{x}=\psi\circ r^1_{\varphi^{-1}(x)}\circ\psi^{-1}.\end{cases}\end{equation}

Using the equations \eqref{l98} and \eqref{l99}, we obtain an action $\operatorname{Aut} \mathfrak{g}\times \operatorname{Aut} \mathfrak{h}$ on $\bigcup\limits_{l,r}Z^2(\mathfrak{g},l,r),$
under which a $2$-cocycle $\omega$ with respect to the pair $(l, r)$ is transformed into a $2$-cocycle $\omega'$ with respect to $(l', r')$ as follows:
\begin{equation*}\omega'(x,y)=\psi \big(\omega(\varphi(x),\varphi(y))\big),\end{equation*}
\begin{equation*}l'_{x}(a)=\psi\big(l_{\varphi(x)} \psi^{-1}(a) \big),\end{equation*}
\begin{equation*}r'_{x}(a)=\psi\big(r_{\varphi(x)} \psi^{-1}(a) \big).\end{equation*}

In this action we say that $\psi$ is an intertwining operator for $l'$ and $l\circ\varphi$ (respectively, $r'$ and $r\circ\varphi$).

Therefore, we conclude that two Leibniz algebras $\widehat{\mathfrak{g}}_1=\mathfrak{g}(\omega^1, l^1, r_1)$ and $\widehat{\mathfrak{g}}_2=\mathfrak{g}(\omega^2,l^2, r^2)$ are isomorphic if and only if
\begin{equation*}{\label{l10}}\omega^2 - \psi \circ \omega^1\circ\varphi \in B^2(\mathfrak{g},l^2,r^2)\end{equation*}
and $\psi$ is an intertwining operator for $l^2$ and $l^1\circ\varphi$ (respectively, $r^2$ and $r^1\circ\varphi$ ). Thus, we obtain following proposition.

\begin{prop}\label{prop2.8} Let $\widehat{\mathfrak{g}}_1=\mathfrak{g}(\omega^1, l^1, r^1)$ and  $\widehat{\mathfrak{g}}_2=\mathfrak{g}(\omega^2,l^2, r^2)$ be extensions of the solvable Leibniz algebra $\mathfrak{g}$ by the abelian algebra $\mathfrak{h}.$
Then the Leibniz algebras $\widehat{\mathfrak{g}}_1$ and $\widehat{\mathfrak{g}}_2$ are isomorphic if and only if $\omega^1$ and $\omega^2$ are in the same  $\operatorname{Aut} \mathfrak{g}\times \operatorname{Aut} \mathfrak{h}$ orbit in $\bigcup\limits_{l,r}H^2(\mathfrak{g},l,r).$

\end{prop}

\section{Extension of solvable Leibniz algebra with null-filiform nilradical}

In this section, we obtain one-dimensional abelian extensions of the solvable Leibniz algebra
 $$R=\begin{cases}
[e_i,e_1]=e_{i+1}, & 1\leq i\leq n-1,\\
[e_i,x]=-ie_i, & 1\leq i \leq n,\\
[x,e_1]=e_1,\end{cases}$$
where $\{x, e_1, e_2, \dots, e_n\}$ is a basis of $R.$

First we give the description of the group of automorphisms of the algebra $R.$

\begin{prop} Any automorphism of the algebra $R$ has the following form:

$$\varphi(x)=x+\sum\limits_{i=1}^n\frac{b^i}{n!}e_i, \quad
\varphi(e_j)=\sum\limits_{i=j}^n\frac{a^jb^{i-j}}{(i-j)!}e_i, \quad  1\leq j \leq n.$$

\end{prop}
\begin{proof} The proof follows directly from the definition of the automorphism.
\end{proof}

Now we present a description of $Z^2(R,l,r),$ i.e., the $2$-cocycle on $R$ with respect to the pair $(l,r)$ by the one dimensional abelian algebra $\mathfrak{h}=\langle e_{n+1} \rangle$.
Since the nilradical of $R$ is $NF_n = \langle e_1, e_2, \ldots, e_n\rangle,$ then from the condition  $NF_n \subseteq \operatorname{ker} l \cap \operatorname{ker} r,$ we obtain that  $$l_x(e_{n+1})=\gamma_1 e_{n+1}, \quad r_x(e_{n+1})=\gamma_2e_{n+1}.$$

Since $l,r$ is a representation, then from \eqref{c3}, we get that $\gamma_1(\gamma_1+\gamma_2)=0.$
Denote
$$d^2_{\omega, l, r}(x,y,z) = \omega([x,y],z)-\omega(x,[y,z])-\omega([x,z],y)-l_x\omega(y,z)-r_y\omega(x,z)+r_z\omega(x,y).$$

\begin{prop}{\label{R}}
A basis of $Z^2(R,l,r)$ is formed by the following cocycles:
\begin{itemize}
\item[1.] If $\gamma_1=0,$ $\gamma_2=-n-1,$ then
\begin{equation}\begin{array}{ll} \label{eq3.1}\omega(e_i,e_1)=b_{i,1}, & 1\leq i\leq n,\\
\omega(x,e_1)=b_{n+1,1},& \omega(e_1,x)=nb_{n+1,1},\\
\omega(e_i,x)=(n+1-i)b_{i-1,1}, & 2\leq i \leq n,\\
\omega(x,x)=b_{n+1,n+1}.
\end{array}\end{equation}

\item[2.] If $\gamma_1=0,$ $\gamma_2\neq -n-1,$  then
\begin{equation}\begin{array}{ll} \label{eq3.2}\omega(e_i,e_1)=b_{i,1}, &1\leq i\leq n-1,\\
\omega(x,e_1)=b_{n+1,1},&  \omega(e_1,x)=-(\gamma_2+1)b_{n+1,1},\\
\omega(e_i,x)=-(i+\gamma_2)b_{i-1.1}, & 2\leq i \leq n,\\
\omega(x,x)=b_{n+1,n+1}.
\end{array}\end{equation}

\item[3.] If  $\gamma_1\neq 0,$ then $\gamma_1=-\gamma_2$ and
\begin{equation}\begin{array}{ll} \label{eq3.3}\omega(e_i,e_1)=b_{i,1}, &1\leq i\leq n-1,\\
 \omega(x,e_1)=b_{n+1,1}, & \omega(e_1,x)=-b_{n+1,1},\\
 \omega(e_i,x)=(\gamma_1-i)b_{i-1,1}, & 2\leq i\leq n,\\
 \omega(x,e_i)=-\gamma_1b_{i-1,1}, & 2\leq i\leq n.
\end{array}\end{equation}
\end{itemize}
\end{prop}
\begin{proof}
Put
 $$\begin{array}{lll}
 \omega(x,x)=b_{n+1,n+1}e_{n+1}, &\omega(e_i,e_j)=b_{i,j}e_{n+1},&1\leq i,j\leq n,\\
 \omega(x,e_i)=b_{n+1,i}e_{n+1}, & \omega(e_i,x)=b_{i,n+1}e_{n+1}, & 1\leq i\leq n.
 \end{array}$$

By utilizing the equation  $\gamma_1(\gamma_1+\gamma_2)=0,$ we examine the following two cases with respect to $\gamma_1:$

\begin{itemize}
\item Let $\gamma_1=0.$ Considering the equation $d^2_{\omega, l, r}(x,y,z) = 0$ for all possible basis elements $\{x, e_1, e_2, \dots, e_n\},$ we obtain
 $$\begin{array}{lllllll}
 d^2_{\omega, l, r}(x,e_1,e_1)=0 & \Rightarrow  & b_{n+1,2}=0,\\
 d^2_{\omega, l, r}(x,e_1,e_i)=0 & \Rightarrow & b_{1,i}=0, & 2\leq i \leq n,\\
 d^2_{\omega, l, r}(x,e_i,e_1)=0 & \Rightarrow& b_{1,i}=-b_{n+1,i+1} & 2\leq i \leq n-1,\\
 d^2_{\omega, l, r}(x,e_n,e_1)=0 & \Rightarrow & b_{1,n}=0,\\
 d^2_{\omega, l, r}(e_i,x, e_1)=0 &\Rightarrow& b_{i+1,n+1}=-(\gamma_2+i+1)b_{i,1},& 1\leq i \leq n-1,\\
 d^2_{\omega, l, r}(e_n, x, e_1)=0 & \Rightarrow &(\gamma_2+n+1)b_{n,1}=0,\\
 d^2_{\omega, l, r}(e_i, x, e_j)=0 &\Rightarrow & (\gamma_2+i)b_{i,j}=0, & 1\leq i \leq n,& 2\leq j \leq n, \\
d^2_{\omega, l, r}(e_i,e_j,x)=0 & \Rightarrow & (\gamma_2+i+j)b_{i,j}=0 &  1\leq i \leq n, & 1\leq j \leq n, \\
d^2_{\omega, l, r}(x,e_1,x)=0& \Rightarrow &b_{1,n+1}=-(\gamma_2+1)b_{n+1,1}, \\
d^2_{\omega, l, r}(x,e_i,x)=0& \Rightarrow &(\gamma_2+i)b_{n+1,i}=0,&2\leq i\leq n, \\
d^2_{\omega, l, r}(e_i,e_1,e_1)=0&\Rightarrow &b_{i,2}=0,& 1\leq i\leq n,\\
d^2_{\omega, l, r}(e_i,e_j,e_1)=0& \Rightarrow &b_{i,j+1}=-b_{i+1,j},& 1\leq i\leq n, & 2\leq j\leq n,\\
d^2_{\omega, l, r}(e_i,e_1,e_j)=0&\Rightarrow &b_{i+1,j}=0,& 1\leq i\leq n, & 2\leq j\leq n.
\end{array}$$

Thus, we get
$$\begin{array}{llll}
b_{1,n+1}=-(\gamma_2+1)b_{n+1,1},&(1+n+\gamma_2)b_{n,1}=0,\\
b_{i,n+1}=-(i+\gamma_2)b_{i-1,1},& 2\leq i \leq n,\\
b_{i,j}=0,& 1\leq i\leq n, \quad 2\leq j\leq n,\\
b_{n+1,i}=0,& 2\leq i\leq n.\\
\end{array}$$
Since $(1+n+\gamma_2)b_{n,1}=0,$ we have that
\begin{itemize}
\item if $\gamma_2=-n-1,$ then $2$-cocycles on $R$ with respect  to the pair $(l,r)$ have  the form \eqref{eq3.1}.
\item if $\gamma_2\neq -n-1,$ then $b_{n,1}=0$ and $2$-cocycles on $R$ with respect to the pair $(l,r)$ have the form \eqref{eq3.2}.

 \end{itemize}

\item Let $\gamma_1\neq 0,$ then $\gamma_1=-\gamma_2$ and from the equation $d^2_{\omega, l, r}(x,y,z) = 0$, we have
$$\begin{array}{lllll}
d^2_{\omega, l, r}(x,e_1,e_1)=0& \Rightarrow &b_{n+1,2}=-\gamma_1 b_{1,1},\\
d^2_{\omega, l, r}(x,e_1,e_i)=0&  \Rightarrow &(1-\gamma_1)b_{1,i}=0,& 2\leq i \leq n,\\
d^2_{\omega, l, r}(x,e_i,e_1)=0& \Rightarrow &b_{n+1,i+1}=-(\gamma_1 b_{i,1}+b_{1,i}),& 2\leq i \leq n-1, \\
d^2_{\omega, l, r}(x,e_n,e_1)=0& \Rightarrow &b_{1,n}=-\gamma_1 b_{n,1},\\
d^2_{\omega, l, r}(x,e_i,e_j)=0& \Rightarrow &\gamma_1 b_{i,j}=0,& 2\leq i \leq n, & 1\leq j \leq n,\\
d^2_{\omega, l, r}(e_i,x,e_1)=0& \Rightarrow &b_{i+1,n+1}=-(1+i-\gamma_1)b_{i,1},& 1\leq i \leq n-1,\\
d^2_{\omega, l, r}(e_n,x,e_1)=0& \Rightarrow &(n+1-\gamma_1)b_{n,1}=0,\\
d^2_{\omega, l, r}(e_1,e_i,x)=0&\Rightarrow &(i+1-\gamma_1)b_{1,i}=0,& 2\leq i\leq n,\\
d^2_{\omega, l, r}(x,x,e_1)=0& \Rightarrow & b_{1,n+1}=-b_{n+1,1},\\
d^2_{\omega, l, r}(x,e_i,x)=0& \Rightarrow & \gamma_1 b_{i,n+1}=(i-\gamma_1)b_{n+1,i}, & 2\leq i\leq n,\\
d^2_{\omega, l, r}(x,x,x)=0& \Rightarrow & \gamma_1 b_{n+1,n+1}=0,\\
d^2_{\omega, l, r}(e_1,e_1,e_i)=0&\Rightarrow & b_{2,i}=0,& 2\leq i\leq n,\\
d^2_{\omega, l, r}(e_i,e_1,e_j)=0&\Rightarrow & b_{i+1,j}=0,& 1\leq i\leq n, & 2\leq j\leq n,\\
d^2_{\omega, l, r}(e_i,e_j,e_1)=0&\Rightarrow &b_{i,j+1}=-b_{i+1,j},& 1\leq i\leq n, &2\leq j\leq n.
\end{array}$$

Therefore, we get
$$\begin{array}{lll}
b_{1,n+1}=-b_{n+1,1}, & b_{n+1,n+1}=0, \\
b_{i,n+1}=(\gamma_1-i)b_{i-1,1}, &2\leq i\leq n,\\
b_{n+1,i}=-\gamma_1b_{i-1,1},& 2\leq i\leq n,\\
b_{i,j}=0,& 1\leq i\leq n, & 2\leq j\leq n.
\end{array}$$

Hence, in this case $2$-cocycles on $R$ with respect  to the pair $(l,r)$ have the form \eqref{eq3.3}.
\end{itemize}
\end{proof}

Now, we determine 2-coboundaries of $R$ with respect to the pair $(l,r).$ Put
$$f(e_i)=c_{i}e_{n+1},\quad \ 1\leq i \leq n, \quad f(x)=c_{n+1}e_{n+1}.$$

For any automorphism $\varphi \in \operatorname{Aut}(R)$ considering
$$df(x,y)=f\big([x,y]\big)-l_{\varphi(x)}f(y)-r_{\varphi(y)}f(x),$$ we have

$$\begin{array}{llll}
df(e_i,e_1)= c_{i+1},  &1\leq i\leq n-1,\\
df(e_i,x)=-(i+\gamma_2)c_i,  &1\leq i\leq n,\\
df(x,e_1)=(1-\gamma_1)c_1,& df(x,x)=-(\gamma_1+\gamma_2)c_{n+1},\\
df(x,e_i)=-\gamma_1c_i, &1\leq i\leq n. \end{array}$$

Thus, we obtain the following corollary.

\begin{cor}\label{cor3.3} We have the following:
\begin{itemize}
\item if $\gamma_1=0,$ $\gamma_2=-n-1$, then  $\operatorname{dim} Z^2(R,l,r) =n+2$ and
$\operatorname{dim} B^2(R,l,r) =n+1;$
\item if $\gamma_1=0,$ $\gamma_2\neq -n-1$, then $\operatorname{dim} Z^2(R,l,r) =n+1$ and
$\operatorname{dim} B^2(R,l,r) =n+1;$
\item if $\gamma_1\neq 0,$ then $\gamma_2=-\gamma_1$ and $\operatorname{dim} Z^2(R,l,r) =n,$
$\operatorname{dim} B^2(R,l,r) =n.$
 \end{itemize}
\end{cor}

We now state the main theorem of this section concerning the one-dimensional abelian extension of the solvable Leibniz algebra $R$ with a null-filiform nilradical.

\begin{thm}\label{R} Let $\widehat{R}$ be a one-dimensional abelian extension of the solvable Leibniz algebra $R,$ then $\widehat{R}$ is isomorphic to the following algebra:
 \begin{equation}\label{R.3.4}\begin{cases}
[e_i,e_1]=e_{i+1}, & 1\leq i\leq n,\\
[x,e_1]=e_1,\\
[e_i,x]=-ie_i, & 1\leq i \leq n+1.
\end{cases} \end{equation}
\end{thm}

\begin{proof}
By Corollary \ref{cor3.3}, we have $H^2(R,l,r) =0$ in the cases $\gamma_1\neq 0$ and $\gamma_1=0,$ $\gamma_2\neq -n-1$.
Therefore, we restrict ourselves to the case of $\gamma_1=0,$ $\gamma_2=-n-1.$ In this case, $\operatorname{dim} H^2(R,l,r) =1$ and an element $\overline{\omega},$ such that $\omega(e_n, e_1) = e_{n+1}$ form a basis of $H^2(R,l,r).$

Note that any automorphism $\psi \in \operatorname{Aut}(\mathfrak{h})$ is given by $\psi(e_{n+1}) = \lambda e_{n+1}.$ Considering the action $\operatorname{Aut} R \times \operatorname{Aut} \mathfrak{h}$ on $\bigcup\limits_{l,r}H^2(R,l,r).$
It is straightforward to verify that any element $\delta\overline{\omega} \in H^2(R,l,r)$ acts on $\overline{\omega'} \in H^2(R,l,r)$, such that
$\omega'(e_n, e_1) = \delta \lambda a^{n+1}e_{n+1},$ with $l'=l,$ $r'=r.$

Since $\lambda\neq 0$, it follows that the orbit of the element $\delta \overline{\omega}$ is the one-dimensional vector space $\langle \overline{\omega}\} \rangle.$ We conclude that any one-dimensional abelian extension of the algebra $R$ defined by the $2$-cocycle $\omega$ with respect to the pair $(l,r)$ is determined up to isomorphism by this orbit, where
$$\omega(e_n,e_1) =  e_{n+1}, \quad l_x(e_{n+1}) = 0, \quad r_x(e_{n+1}) = -(n+1) e_{n+1}.$$

We now define the multiplication in the algebra $\widehat{R}=R\oplus\{e_{n+1}\}$ by
    $$[e_n,e_1]=\omega(e_n,e_1) =  e_{n+1}, \quad [e_{n+1},x]=r_x(e_{n+1})=-(n+1)e_{n+1} .$$

Therefore, combining this with the multiplication in $R$, we obtain the multiplication \eqref{R.3.4}.
\end{proof}

\section{Extension of $5$-dimensional solvable Leibniz algebra with $3$-dimensional nilradical}

In this section, we classify the one-dimensional abelian extensions of five-dimensional solvable Leibniz algebras whose nilradical is three-dimensional. It is known that there exist exactly four five-dimensional solvable Leibniz algebras with a non-trivial three-dimensional nilradical
\cite{Canete}.
\begin{longtable}{ll}
\hline
$H:$ & $\begin{array}{lllll} [e_1,e_2]=e_3,& [e_2,e_1]=-e_3, &[e_1,x_1]=e_1, & [x_1,e_1]=-e_1, & [e_3,x_1]=e_3,\\[1mm]
 [x_1,e_3]=-e_3, & [e_2,x_2]=e_2,& [x_2,e_2]-e_2, & [e_3,x_2]=e_3, & [x_2, e_3]=-e_3,\end{array}$
 \\ \hline
 $L_1:$ & $\begin{array}{llllll} [e_2,e_1]=e_3,& [e_1,x_1]=e_1, & [x_1,e_1]=-e_1, & [e_3,x_1]=e_3,&  [e_2,x_2]=e_2, & [e_3,x_2]=e_3, \end{array}$
 \\ \hline
  $L_2:$ & $\begin{array}{lllll} [e_1,e_1]=e_3,& [e_1,x_1]=e_1, & [x_1,e_1]=-e_1, & [e_3,x_1]=2e_3,&  [e_2,x_2]=e_2, \end{array}$ \\ \hline
  $L_3:$ & $\begin{array}{llllll} [e_1,e_1]=e_3,& [e_1,x_1]=e_1, & [x_1,e_1]=-e_1, &[e_3,x_1]=2e_3,&  [e_2,x_2]=e_2, & [x_2,e_2]=-e_2. \end{array}$
 \\ \hline
\end{longtable}

\newpage
 The following proposition provides a description of the automorphism groups of these algebras.
 \begin{prop}\label{Aut3-dim}
 Any automorphism of the algebras $H,$ $L_1,$ $L_2$ and $L_3$ has the following form:
$$\begin{array}{ll}\operatorname{Aut}(H):&\begin{array}{l}\left\{\begin{array}{llll}\varphi_1(e_1)=a_{1}e_1+a_{2}e_3, & \varphi_1(x_1)=x_1+\frac{a_{4}}{a_{3}}e_1+\bigl(\frac{a_{2}a_{4}}{a_{1}a_{3}}+a_{5}\bigr)e_3,\\ [1mm] \varphi_1(e_2)=a_{3}e_2+a_{4}e_3,& \varphi_1(x_2)=x_2-\frac{a_{2}}{a_{1}}e_2+a_{5}e_3,\\ [1mm]
\varphi_1(e_3)=a_{1}a_{3}e_3,\end{array}\right.\\[6mm]
\left\{\begin{array}{llll}\varphi_2(e_1)=a_{1}e_2+a_{2}e_3, & \varphi_2(x_1)=x_2+\frac{a_{4}}{a_{3}}e_2+\bigl(\frac{a_{2}a_{4}}{a_{1}a_{3}}+a_{5}\bigr)e_3,\\[1mm]
\varphi_2(e_2)=-a_{3}e_1+a_{4}e_3,&  \varphi_2(x_2)=x_1+\frac{a_{2}}{a_{1}}e_1+a_{5}e_3,\\[1mm]
\varphi_2(e_3)=a_{1}a_{3}e_3,\\\end{array}\right.\\[7mm]\end{array}\\
\operatorname{Aut}(L_1):&\begin{array}{lllll}\varphi(e_1)=a_{1}e_1, & \varphi(e_2)=a_{2}e_2,& \varphi(e_3)=a_{1}a_{2}e_3,&
\varphi(x_1)=x_1, & \varphi(x_2)=x_2,\end{array}\\[2mm]
\operatorname{Aut}(L_2):& \left\{\begin{array}{lllll}\varphi(e_1)=a_{1}e_1-a_{1}a_{3}e_3, & \varphi(x_1)=x_1+a_{3}e_1-\frac{1}{2}a_{3}^2e_3 ,\\[1mm]
\varphi(e_2)=a_{2}e_2, & \varphi(x_2)=x_2, \\ [1mm]
\varphi(e_3)=a_{1}^2e_3,\end{array}\right.\\ [7mm]
\operatorname{Aut}(L_3):& \left\{\begin{array}{lllll}\varphi(e_1)=a_{1}e_1-a_{1}a_{3}e_3, &\varphi(x_1)=x_1+a_{3}e_1-\frac{1}{2}a_{3}^2e_3,\\[1mm]
\varphi(e_2)=a_{2}e_2,&  \varphi(x_2)=x_2+a_{4}e_2,\\[1mm]
\varphi(e_3)=a_{1}^2e_3.\end{array}\right.\end{array}$$

\end{prop}
 \begin{proof} The proof follows directly from straightforward computation.
\end{proof}

Now we describe all $2$-coboundaries on $\mathfrak{g}$ with respect to the pair $(l,r),$
by the one-dimensional abelian algebra $\mathfrak{h}=\langle e_4\rangle,$
where $\mathfrak{g}$ one of the algebras $H,$ $L_1,$ $L_2$ and $L_3.$

For the basis elements $\{e_1, e_2, e_3, x_1, x_2\}$ and $f \in \operatorname{Hom}(\mathfrak{g},\mathfrak{h})$, we put
$$f(e_i)=c_{i}e_{4},\quad \ 1\leq i \leq 3, \quad f(x_1)=c_{4}e_{4}, \quad f(x_2)=c_{5}e_{4}.$$

Note that nilradical of these algebras is $N=\langle e_1, e_2, e_3\rangle,$
Since we consider $l,r: \mathfrak{g} \rightarrow \operatorname{End}(\mathfrak{h}) $ with the condition $N\subseteq \operatorname{ker} l \cap \operatorname{ker} r,$  it follows that
 $$\begin{array}{ll}l_{x_1}(e_{4})=\alpha_1 e_{4}, & r_{x_1}(e_{4})=\alpha_2e_{4},\\[1mm]
l_{x_2}(e_{4})=\beta_1 e_{4}, & r_{x_2}(e_{4})=\beta_2e_{4}.\end{array}$$

For any automorphism $\varphi \in \operatorname{Aut}(\mathfrak{g}),$ consider
$$df(x,y)=f\big([x,y]\big)-l_{\varphi (x)} f(y)-r_{\varphi(y)} f(x),$$ and we obtain the following result.

\begin{prop}\label{B2} Any $2$-coboundary with respect to the pair $(l,r)$ for the algebras $H,L_1,L_2$ and $L_3$ has the following forms:
$$B^2(H,l,r):\begin{array}{llll}
df(e_1,e_2)= c_{3},  & df(e_1,x_1)=(1-\alpha_2)c_1,& df(e_1,x_2)=-\beta_2c_1,\\
  df(e_2,e_1)= -c_{3},& df(e_2,x_1)=-\alpha_2c_{2}, & df(e_2,x_2)=(1-\beta_2)c_2,\\
 df(e_3,x_1)=(1-\alpha_2) c_{3}, & df(e_3,x_2)=(1-\beta_2)c_{3},&
 df(x_1,e_1)=-(1+\alpha_1)c_{1},\\
 df(x_1,e_2)=-\alpha_1 c_{2}, & df(x_1,e_3)=-(1+\alpha_1)c_{3}, &  df(x_1,x_1)=-(\alpha_1+\alpha_2) c_{4}, \\
  df(x_1,x_2)=-\alpha_1c_{5}-\beta_2c_4,&
df(x_2,e_1)=-\beta_1c_{1},  & df(x_2,e_2)=-(1+\beta_1) c_{2},\\
df(x_2,e_3)=-(1+\beta_1)c_{3},&
df(x_2,x_1)=-\beta_1c_{4}-\alpha_2c_5, & df(x_2,x_2)=-(\beta_1+\beta_2) c_{5}. \end{array}$$

$$B^2(L_1,l,r):\begin{array}{lllll}df(e_1,x_1)=(1-\alpha_2)c_1,  &df(e_1,x_2)=-\beta_2c_1,&df(e_2,e_1)= c_{3}, \\ df(e_2,x_1)=-\alpha_2c_{2},&
df(e_2,x_2)=(1-\beta_2)c_2, & df(e_3,x_1)=(1-\alpha_2) c_{3}, \\
 df(e_3,x_2)=(1-\beta_2)c_{3},& df(x_1,e_1)=-(1+\alpha_1)c_{1},&
df(x_1,e_2)=-\alpha_1 c_{2},\\
 df(x_1,e_3)=-\alpha_1c_{3},& df(x_1,x_1)=-(\alpha_1+\alpha_2) c_{4}, & df(x_1,x_2)=-\alpha_1c_{5}-\beta_2c_4,\\
df(x_2,e_1)=-\beta_1c_{1}, &
df(x_2,e_2)=-\beta_1 c_{2}, & df(x_2,e_3)=-\beta_1c_{3},\\
df(x_2,x_1)=-\beta_1c_{4}-\alpha_2c_5, &
df(x_2,x_2)=-(\beta_1+\beta_2) c_{5}.\end{array}$$

$$B^2(L_2,l,r):\begin{array}{llll}
df(e_1,e_1)= c_{3}, & df(e_1,x_1)=(1-\alpha_2)c_1,  &df(e_1,x_2)=-\beta_2c_1,\\
df(e_2,x_1)=-\alpha_2c_{2},& df(e_2,x_2)=(1-\beta_2)c_2, & df(e_3,x_1)=(2-\alpha_2) c_{3}, \\ df(e_3,x_2)=-\beta_2c_{3}, &
df(x_1,e_1)=-(1+\alpha_1)c_{1},& df(x_1,e_2)=-\alpha_1 c_{2},\\
df(x_1,e_3)=-\alpha_1c_{3},& df(x_1,x_1)=-(\alpha_1+\alpha_2) c_{4}, & df(x_1,x_2)=-\alpha_1c_{5}-\beta_2c_4,\\
df(x_2,e_1)=-\beta_1c_{1}, &
df(x_2,e_2)=-\beta_1 c_{2}, &df(x_2,e_3)=-\beta_1c_{3},\\
df(x_2,x_1)=-\beta_1c_{4}-\alpha_2c_5, & df(x_2,x_2)=-(\beta_1+\beta_2) c_{5}.\end{array}$$

$$B^2(L_3,l,r): \begin{array}{llll}
df(e_1,e_1)= c_{3}, & df(e_1,x_1)=(1-\alpha_2)c_1,  & df(e_1,x_2)=-\beta_2c_1, \\
df(e_2,x_1)=-\alpha_2c_{2}, & df(e_2,x_2)=(1-\beta_2)c_2, & df(e_3,x_1)=(2-\alpha_2) c_{3}, \\
df(e_3,x_2)=-\beta_2c_{3},& df(x_1,e_1)=-(1+\alpha_1)c_{1}, &
df(x_1,e_2)=-\alpha_1 c_{2},\\
df(x_1,e_3)=-\alpha_1c_{3},& df(x_1,x_1)=-(\alpha_1+\alpha_2) c_{4}, & df(x_1,x_2)=-\alpha_1c_{5}-\beta_2c_4,\\
 df(x_2,e_1)=-\beta_1c_{1},  & df(x_2,e_2)=-(1+\beta_1) c_{2}, &  df(x_2,e_3)=-\beta_1c_{3},\\
df(x_2,x_1)=-\beta_1c_{4}-\alpha_2c_5, & df(x_2,x_2)=-(\beta_1+\beta_2) c_{5}. \end{array}$$

\end{prop}

 \begin{proof} The proof follows directly from straightforward computation.
\end{proof}

\subsection{Extensions of the algebra $H$}

Now, using the algorithm for constructing of solvable Leibniz algebras, we derive all possible extensions of the solvable Leibniz algebra $H$
by the one-dimensional abelian algebra $\mathfrak{h}=\langle e_4 \rangle$.
The condition \eqref{c3}, to the linear maps $l,r: H \rightarrow \operatorname{End}(\mathfrak{h})$ such that
$$\begin{array}{ll}l_{x_1}(e_{4})=\alpha_1 e_{4}, & r_{x_1}(e_{4})=\alpha_2e_{4},\\[1mm]
l_{x_2}(e_{4})=\beta_1 e_{4}, & r_{x_2}(e_{4})=\beta_2e_{4}.\end{array}$$
leads to the following system of equalities:
\begin{equation}\label{Hz}\begin{array}{ll} \alpha_1(\alpha_1+\alpha_2)=0, & \alpha_1(\beta_1+\beta_2)=0,\\[1mm]
\beta_1(\alpha_1+\alpha_2)=0,& \beta_1(\beta_1+\beta_2)=0.\end{array}\end{equation}

In the following Proposition, we describe all $2$-cocycles on $H$ with respect to the pair $(l,r),$ by the one dimensional abelian algebra $\mathfrak{h}=\langle e_4 \rangle$.

\begin{prop}{\label{Z2forH}}
A basis of $Z^2(H,l,r)$ consists of the following cocycles:
\begin{itemize}
\item[I.] $\alpha_1=\beta_1=\beta_2=0,$ $\alpha_2=1:$
\begin{equation}\label{eq4.2}\begin{array}{lllll}\omega(e_1,e_2)=b_{1,2},  & \omega(e_2,e_1)=-b_{1,2}, &  \omega(e_2,x_1)=b_{2,4}, & \omega(e_2,x_2)=-b_{2,4},\\ [1mm]
 \omega(e_3,x_2)=b_{1,2},&\omega(x_1,e_1)=b_{4,1},&\omega(x_1,e_3)=-b_{1,2},&\omega(x_1,x_1)=b_{4,4}, \\ [1mm]
 \omega(x_2,e_1)=b_{5,1} & \omega(x_2,e_2)=b_{2,4}, &\omega(x_2,e_3)=-b_{1,2}, &
\omega(x_2,x_1)=b_{5,4}.\end{array}\end{equation}

\item[II.] $\alpha_1=\alpha_2=\beta_1=0,$ $\beta_2=1:$
\begin{equation}\label{eq4.3}\begin{array}{lllllll}\omega(e_1,e_2)=b_{1,2},  & \omega(e_1,x_1)=b_{1,4}, & \omega(e_1,x_2)=-b_{1,4},& \omega(e_2,e_1)=-b_{1,2}, \\ [1mm]
 \omega(e_3,x_1)=b_{1,2},&\omega(x_1,e_1)=-b_{1,4},& \omega(x_1,e_2)=b_{4,2}, &
  \omega(x_1,e_3)=-b_{1,2}, \\ [1mm] \omega(x_1,x_2)=b_{4,5}, &
 \omega(x_2,e_2)=b_{5,2} & \omega(x_2,e_3)=-b_{1,2}, &
\omega(x_2,x_2)=b_{5,5}.\end{array}\end{equation}

\item[III.] $\alpha_1=\beta_1=\beta_2=0,$ $\alpha_2=2:$
\begin{equation}\label{eq4.4}\begin{array}{lllll}\omega(e_1,e_1)=b_{1,1}, & \omega(e_1,e_2)=b_{1,2},  & \omega(e_1,x_2)=b_{1,5},& \omega(e_2,e_1)=-b_{1,2},  \\ [1mm]
\omega(e_2,x_1)=2b_{2,5},& \omega(e_2,x_2)=b_{2,5},& \omega(e_3,x_1)=-b_{1,2}, & \omega(e_3,x_2)=b_{1,2}, \\ [1mm]
\omega(x_1,e_1)=b_{1,5}, & \omega(x_1,e_3)=-b_{1,2}, &  \omega(x_1,x_1)=b_{4,4}, &\\ [1mm]
\omega(x_2,e_2)=-b_{2,5}, &  \omega(x_2,e_3)=-b_{1,2}, &   \omega(x_2,x_1)=b_{5,4}.\end{array}\end{equation}

\item[IV.] $\alpha_1=\alpha_2=\beta_1=0,$ $\beta_2=2:$
\begin{equation}\label{eq4.5}\begin{array}{lllll} \omega(e_1,e_2)=b_{1,2}, & \omega(e_1,x_1)=b_{1,4},& \omega(e_1,x_2)=-2b_{1,4},&
 \omega(e_2,e_1)=-b_{1,2},  \\ [1mm] \omega(e_2,e_2)=b_{2,2}, & \omega(e_2,x_2)=b_{2,5},&
 \omega(e_3,x_1)=b_{1,2}, & \omega(e_3,x_2)=-b_{1,2}, & \\ [1mm]
\omega(x_1,e_1)=-b_{1,4}, &  \omega(x_1,e_3)=-b_{1,2}, &   \omega(x_1,x_2)=b_{4,5}, \\[1mm]
\omega(x_2,e_2)=b_{2,5}, & \omega(x_2,e_3)=-b_{1,2}, &  \omega(x_2,x_2)=b_{5,5}.\end{array}\end{equation}

\item[V.] $\alpha_1=\beta_1=0,$ $\alpha_2=\beta_2=1:$
\begin{equation}\label{eq4.6}\begin{array}{llll}\omega(e_1,e_2)=b_{1,2},  & \omega(e_1,x_2)=b_{1,5}, &\omega(e_2,e_1)=b_{2,1},& \omega(e_2,x_1)=b_{2,4},\\ [1mm]
\omega(x_1,e_1)=b_{1,5}, & \omega(x_1,e_3)=-b_{1,2},& \omega(x_1,x_1)=b_{4,4},   &\omega(x_1,x_2)=b_{4,4},\\ [1mm]
\omega(x_2,e_2)=b_{2,4},  &  \omega(x_2,e_3)=b_{2,1},& \omega(x_2,x_1)=b_{5,4},& \omega(x_2,x_2)=b_{5,4}.\end{array}\end{equation}

 \item[VI.] $\alpha_1=\beta_1=0,$ $ (\alpha_2,\beta_2)\neq (1,0),(0,1),(2,0),(0,2),(1,1):$
\begin{equation}\label{eq4.7}\begin{array}{llll} \omega(e_1,e_2)=b_{1,2},  & \omega(e_1,x_1)=(\alpha_2-1)b_{4,1},&
\omega(e_1,x_2)=\beta_2b_{4,1}, \\ [1mm]
\omega(e_2,e_1)=-b_{1,2},  & \omega(e_2,x_1)=\alpha_2b_{5,2}, & \omega(e_2,x_2)=(\beta_2-1)b_{5,2}, \\ [1mm]
\omega(e_3,x_1)=(1-\alpha_2)b_{1,2},& \omega(e_3,x_2)=(1-\beta_2)b_{1,2},& \omega(x_1,e_1)=b_{4,1}, \\ [1mm]
\omega(x_1,e_3)=-b_{1,2}, & \omega(x_1,x_1)=\alpha_2b_{4,4},  &\omega(x_1,x_2)=\beta_2b_{4,4},\\[1mm]
 \omega(x_2,e_2)=b_{5,2},& \omega(x_2,e_3)=-b_{1,2},& \omega(x_2,x_1)=\alpha_2b_{5,4}, \\ [1mm]
 \omega(x_2,x_2)=\beta_2b_{5,4} .\end{array}\end{equation}
\item[VII.] $\alpha_1=-\alpha_2=-1,$ $\beta_1=\beta_2=0:$
\begin{equation}\label{eq4.8}\begin{array}{llll}
\omega(e_1,e_2)=b_{1,2}, & \omega(e_1,x_1)=b_{1,4},& \omega(e_1,x_2)=b_{1,5}, & \omega(e_2,e_1)=-b_{1,2}, \\ [1mm]
 \omega(e_2,x_1)=b_{2,4}, &  \omega(e_2,x_2)=-b_{2,4}, &  \omega(e_3,x_2)=b_{1,2}, & \omega(x_1,e_1)=-b_{1,4},\\ [1mm]
\omega(x_1,e_2)=-b_{2,4},&  \omega(x_1,x_2)=b_{4,5},& \omega(x_2,e_1)=-b_{1,5},&  \omega(x_2,e_2)=b_{2,4}, \\ [1mm]
\omega(x_2,e_3)=-b_{1,2},  & \omega(x_2,x_1)=-b_{4,5}.\end{array}\end{equation}
\item[VIII.] $\alpha_1=\alpha_2=0,$ $\beta_1=-\beta_2=-1:$
\begin{equation}\label{eq4.9}\begin{array}{llll}
\omega(e_1,e_2)=b_{1,2}, & \omega(e_1,x_1)=b_{1,4},& \omega(e_1,x_2)=-b_{1,4}, & \omega(e_2,e_1)=-b_{1,2}, \\ [1mm]
 \omega(e_2,x_1)=b_{2,4}, &  \omega(e_2,x_2)=b_{2,5}, &  \omega(e_3,x_1)=b_{1,2}, & \omega(x_1,e_1)=-b_{1,4},\\ [1mm]
\omega(x_1,e_2)=-b_{2,4},& \omega(x_1,e_3)=-b_{1,2},  & \omega(x_1,x_2)=b_{4,5},& \omega(x_2,e_1)=-b_{1,5},\\ [1mm] \omega(x_2,e_2)=-b_{2,5}, & \omega(x_2,x_1)=-b_{4,5}.\end{array}\end{equation}

\item[IX.] $\alpha_1=-\alpha_2=-1,$ $\beta_1=-\beta_2=-2:$
\begin{equation}\label{eq4.10}\begin{array}{llll}
\omega(e_1,e_2)=b_{1,2}, & \omega(e_1,x_2)=b_{1,5}, &\omega(e_2,e_1)=-b_{1,2}, & \omega(e_2,e_3)=b_{2,3}, \\ [1mm]
\omega(e_2,x_1)=b_{2,4}, & \omega(e_2,x_2)=b_{2,4}, & \omega(e_3,e_2)=-b_{2,3},& \omega(e_3,x_2)=-b_{1,2},  \\ [1mm]
\omega(x_1,e_2)=-b_{2,4},& \omega(x_1,x_2)=b_{4,5}, & \omega(x_2,e_1)=-b_{1,5},& \omega(x_2,e_2)=-b_{2,4}, \\ [1mm]
\omega(x_2,e_3)=b_{1,2}, & \omega(x_2,x_1)=-b_{4,5}.\end{array}\end{equation}

\item[X.] $\alpha_1=-\alpha_2=-2,$ $\beta_1=-\beta_2=-1:$
\begin{equation}\label{eq4.11}\begin{array}{llll}
\omega(e_1,e_2)=b_{1,2}, & \omega(e_1,e_3)=b_{1,3}, & \omega(e_1,x_1)=b_{1,4}, & \omega(e_1,x_2)=b_{1,4},\\ [1mm]
\omega(e_2,e_1)=-b_{1,2}, &  \omega(e_2,x_1)=b_{2,4}, &  \omega(e_3,e_1)=-b_{1,3},& \omega(e_3,x_1)=-b_{1,2},\\ [1mm]
\omega(x_1,e_1)=-b_{1,4},&  \omega(x_1,e_2)=-b_{2,4},& \omega(x_1,e_3)=b_{1,2}, &
\omega(x_1,x_2)=b_{4,5}, \\ [1mm]
\omega(x_2,e_1)=-b_{1,4}, & \omega(x_2,x_1)=-b_{4,5}.\end{array}\end{equation}

\item[XI.] $\alpha_1=-\alpha_2$ $\beta_1=-\beta_2,$ $(\alpha_1, \beta_1) \neq (-1,0), (0,-1), (-2,-1), (-1,-2):$
\begin{equation}\label{eq4.12}\begin{array}{llll}\omega(e_1,e_2)=b_{1,2}, & \omega(e_1,x_1)=(1+\alpha_1)b_{1,4}, &\omega(e_1,x_2)=\beta_1b_{1,4},\\ [1mm]
 \omega(e_2,e_1)=-b_{1,2},&
  \omega(e_2,x_1)=\alpha_1b_{2,4}, & \omega(e_2,x_2)=(1+\beta_1) b_{2,4},  \\ [1mm]
  \omega(e_3,x_1)=(1+\alpha_1)b_{1,2}, & \omega(e_3,x_2)=(1+\beta_1)b_{1,2}, &\omega(x_1,e_1)=-(1+\alpha_1)b_{1,4}, \\ [1mm] \omega(x_1,e_2)=-\alpha_1b_{2,4},& \omega(x_1,e_3)=-(1+\alpha_1)b_{1,2}, &  \omega(x_1,x_2)=b_{4,5},\\ [1mm] \omega(x_2,e_1)=-\beta_1b_{1,4},& \omega(x_2,e_2)=-(1+\beta_1) b_{2,4},
& \omega(x_2,e_3)=-(1+\beta_1)b_{1,2},\\ [1mm]  \omega(x_2,x_1)=-b_{4,5}.\end{array}\end{equation}
\end{itemize}

\end{prop}
\begin{proof} For any $\omega\in Z^2(H,l,r),$
let us denote $\omega(x,y)=b_{i,j}e_4,$ $1\leq i,j\leq 5$, where $x,y\in \{e_1,e_2,e_3,x_1,x_2\}.$
Using the equality \eqref{c4}, we compute the cocycles and obtain the following relations:
$$\small{\begin{array}{llll}
\beta_1b_{1,1}=0, & \beta_2b_{1,1}=0, & \alpha_1b_{1,1}=0, & (2-\alpha_2)b_{1,1}=0,\\[1mm]
b_{1,3}=-b_{3,1},&  b_{2,3}=-b_{3,2}, & b_{3,3}=0,\\[1mm]
(2+\alpha_1)b_{1,3}=0, & (2-\alpha_2)b_{1,3}=0, & (1+\beta_1)b_{1,3}=0, &(1-\beta_2)b_{1,3}=0,\\[1mm]
\beta_1b_{1,4}=(\alpha_2-1)b_{5,1},&\beta_2b_{1,4}=(\alpha_2-1)b_{1,5},\\[1mm]
b_{1,4}=(\alpha_1+\alpha_2-1)b_{4,1},& (\alpha_1+1)b_{1,4}=(\alpha_2-1)b_{4,1},\\[1mm]
\beta_1b_{1,5}=\beta_2b_{5,1},& b_{1,5}=\alpha_1b_{5,1}+\beta_2b_{4,1},\\[1mm]
\alpha_1b_{2,2}=0, & \alpha_2b_{2,2}=0,& \beta_1b_{2,2}=0, & (2-\beta_2)b_{2,2}=0,\\[1mm]
(1-\alpha_2)b_{2,3}=0,& (2-\beta_2)b_{2,3}=0,& (\alpha_1+1)b_{2,3}=0, &(2+\beta_1)b_{2,3}=0, \\[1mm]
\alpha_1b_{2,4}=\alpha_2b_{4,2},& b_{2,4}=\beta_1b_{4,2}+\alpha_2b_{5,2}, & (1+\beta_1)b_{2,4}=\alpha_2b_{5,2},\\[1mm]
\alpha_1b_{2,5}=(\beta_2-1)b_{4,2},& \alpha_2b_{2,5}=(\beta_2-1)b_{2,4},\\[1mm]
(\beta_1+1)b_{2,5}=(\beta_2-1)b_{5,2},& b_{2,5}=(\beta_1+\beta_2-1)b_{5,2},\\[1mm]
b_{3,4}=(1-\alpha_2)b_{1,2},& b_{3,4}=(\alpha_2-1)b_{2,1},\\[1mm] (\alpha_1+1)b_{3,4}=(\alpha_2-1)b_{4,3},&b_{3,4}=(\alpha_1+\alpha_2-1)b_{4,3},\\[1mm]
(\beta_1+1)b_{3,4}=(\alpha_2-1)b_{5,3}, & b_{3,4}=\beta_1b_{4,3}+(\alpha_2-1)b_{5,3},\\[1mm]
 b_{3,5}=(1-\beta_2)b_{1,2},& b_{3,5}=(\beta_2-1)b_{2,1}, &  (\alpha_1+1)b_{3,5}=(\beta_2-1)b_{4,3},\\[1mm] b_{3,5}=\alpha_1b_{5,3}+(\beta_2-1)b_{4,3}, & (\beta_1+1)b_{3,5}=(\beta_2-1)b_{5,3}, & b_{3,5}=(\beta_1+\beta_2-1)b_{5,3},\\[1mm]
\beta_2b_{4,1}=(\alpha_1+1)b_{1,5},& \beta_1b_{4,1}=(1-\alpha_2)b_{5,1}, \\[1mm]
(1-\beta_2)b_{4,2}=\alpha_1b_{5,2}, & (\alpha_1+\alpha_2)b_{4,2}=0,\\[1mm]
 b_{4,3}=-(\alpha_1+1)b_{1,2},&  b_{4,3}=\alpha_1b_{2,1}-b_{1,2},\\[1mm]
 \alpha_1b_{4,4}=0, & \beta_1b_{4,4}=0,&  \alpha_1b_{5,5}=0, & \beta_1b_{5,5}=0,  \\[1mm]
 (\alpha_1+\alpha_2)b_{4,5}=\beta_2b_{4,4},&  b_{5,3}=b_{2,1}-\beta_1b_{1,2},& b_{5,3}=(1+\beta_1)b_{2,1},\\[1mm]
(\beta_1+\beta_2)b_{5,1}=0,& \alpha_2b_{5,5}=(\beta_1+\beta_2)b_{5,4},\\[1mm]
\beta_2b_{4,4}=\alpha_2b_{4,5}-\alpha_1b_{5,4}, &\alpha_2b_{5,5}=\beta_2b_{5,4}-\beta_1b_{4,5}. \end{array}}$$

 By \eqref{Hz}, we derive that $\alpha_1=\beta_1=0$ or $\alpha_1=-\alpha_2$ and $\beta_1=-\beta_2.$ Thus,
 to determine the values of $b_{i,j}$ we consider the following two cases:

\begin{itemize}
\item[Case 1.]Let $\alpha_1=\beta_1=0.$  Then we have
$$\begin{array}{lll}
b_{1,3}=0,& b_{1,4}= (\alpha_2-1)b_{4,1}, & b_{1,5}= \beta_2b_{4,1}, \\[1mm]
b_{2,3}=0, & b_{2,4}= \alpha_2b_{5,2}, & b_{2,5}= (\beta_2-1)b_{5,2}, \\[1mm]
b_{3,1}=0,& b_{3,2}=0, & b_{3,3}=0,\\[1mm]
b_{3,4}= (1-\alpha_2)b_{1,2}, & b_{3,5}= (\beta_2-1)b_{2,1}, \\[1mm]
b_{4,3}= - b_{1,2}, & b_{5,3}= b_{2,1}, \\[1mm]
 \end{array}$$
and the following restrictions
$$\begin{array}{llll}
 \beta_2b_{1,1}=0, & (\alpha_2-2)b_{1,1}=0,& (\alpha_2-1)(b_{1,2}+b_{2,1})=0, \\[1mm]
 \alpha_2b_{2,2}=0,& (\beta_2-2)b_{2,2}=0,&(\beta_2-1)(b_{1,2}+b_{2,1})=0,\\[1mm]
 \alpha_2b_{4,2}=0,& (\beta_2-1)b_{4,2}=0,& \alpha_2b_{4,5}=\beta_2b_{4,4},\\[1mm]
\beta_2b_{5,1}=0, & (\alpha_2-1)b_{5,1}=0,& \alpha_2b_{5,5}=\beta_2b_{5,4}.\\[1mm]
 \end{array}$$

Now, consider the following subcases:

\begin{itemize}
\item[1.1.] If $\alpha_2=1,$ $\beta_2=0,$ then we obtain $b_{2,1}=-b_{1,2}$ and
$b_{1,1}=b_{2,2}=b_{4,2}=b_{4,5}= b_{5,5}=0.$ Hence, we conclude that any $2$-cocycle on $H$ with respect to the pair $(l,r)$ is given by the expression in \eqref{eq4.2}.

\item[1.2.] If $\alpha_2=0,$ $\beta_2=1,$ then we obtain $b_{2,1}=-b_{1,2}$ and
$b_{1,1}=b_{2,2}=b_{4,4}=b_{5,1}= b_{5,4}=0.$ Hence, we get that any $2$-cocycle on $H$ with respect to the pair $(l,r)$ has the form \eqref{eq4.3}.

  \item[1.3.] If $\alpha_2=2,$ $\beta_2=0,$ then we obtain $b_{2,1}=-b_{1,2}$ and
$b_{2,2}=b_{4,2}=b_{4,5}= b_{5,1}= b_{5,5}=0.$ Hence, any $2$-cocycle on $H$ with respect to the pair $(l,r)$ has the form \eqref{eq4.4}.

\item[1.4.] If $\alpha_2=0,$ $\beta_2=2,$ then we obtain $b_{2,1}=-b_{1,2}$ and
$b_{1,1}=b_{4,2}=b_{4,4}= b_{5,1}= b_{5,4}=0.$ Hence, we get that any $2$-cocycle on $H$ with respect to the pair $(l,r)$ has the form \eqref{eq4.5}.

\item[1.5.] If $\alpha_2=1,$ $\beta_2=1,$ then we obtain $b_{4,5}=b_{4,4},$ $b_{5,5}=b_{5,4}$   and
$b_{1,1}=b_{2,2}=b_{4,2}= b_{5,1}=0.$ Hence, we get that any $2$-cocycle on $H$ with respect to the pair $(l,r)$ has the form \eqref{eq4.6}.

 \item[1.6.] If $ (\alpha_2,\beta_2)\neq (1,0),(0,1),(2,0),(0,2),(1,1), $ then we obtain $b_{2,1}=-b_{1,2},$  $\alpha_2b_{4,5}=\beta_2 b_{4,4},$ $\alpha_2b_{5,5}=\beta_2b_{5,4}$   and
$b_{1,1}=b_{2,2}=b_{4,2}= b_{5,1}=0.$ Hence, we get that any $2$-cocycle on $H$ with respect to the pair $(l,r)$ has the form \eqref{eq4.7}.
\end{itemize}

 \item [Case 2.] Let $\alpha_1=-\alpha_2$ and $\beta_1=-\beta_2.$  Then we have
$$\begin{array}{lllll}
b_{1,1}=0, & b_{2,1}=-b_{1,2},& b_{2,2}=0, & b_{3,1}=-b_{1,3}, \\[1mm]
  b_{3,2}=-b_{2,3}, & b_{3,3}=0,  & b_{3,4}= (1+\alpha_1)b_{1,2}, & b_{3,5}=(1+\beta_1)b_{1,2}, \\[1mm]
  b_{4,1}=-b_{1,4}, & b_{4,2}=-b_{2,4},& b_{4,3}= -b_{3,4}, & b_{4,4}=0, \\[1mm]
 b_{5,1}=-b_{1,5},& b_{5,2}= -b_{2,5}, &
 b_{5,3}=- b_{3,5},& b_{5,4}=-b_{4,5}, & b_{5,5}=0.
 \end{array}$$
and the following restrictions
$$\begin{array}{llll}
 (\alpha_1+2)b_{1,3}=0,& (\beta_1+1)b_{1,3}=0, & \beta_1b_{1,4}=(1+\alpha_1)b_{1,5}, \\[1mm]
 (\alpha_1+1)b_{2,3}=0,& (\beta_1+2)b_{2,3}=0,&  \alpha_1b_{2,5}=(1+\beta_1)b_{2,4}.
 \end{array}$$

Now, consider the following subcases:
\begin{itemize}
\item[2.1.] If $\alpha_1=-1,$ $\beta_1=0,$ then we obtain $b_{2,5}=-b_{2,4}$ and
$b_{1,3}=b_{2,3}=0.$ Hence, we get that any $2$-cocycle on $H$ with respect to the pair $(l,r)$ has the form \eqref{eq4.8}.

\item[2.2.] If $\alpha_1=0,$ $\beta_1=-1,$ then we obtain $b_{1,5}=-b_{1,4}$ and
$b_{1,3}=b_{2,3}=0.$ Hence, we get that any $2$-cocycle on $H$ with respect to the pair $(l,r)$ has the form \eqref{eq4.9}.

\item[2.3.] If $\alpha_1=-1,$ $\beta_1=-2,$ then we obtain $b_{2,5}=b_{2,4}$ and
$b_{1,3}=b_{1,4}=0.$ Hence, we get that any $2$-cocycle on $H$ with respect to the pair $(l,r)$ has the form \eqref{eq4.10}.

\item[2.4.] If $\alpha_1=-2,$ $\beta_1=-1,$ then we obtain $b_{1,5}=b_{1,4}$ and
$b_{2,3}=b_{2,5}=0.$ Hence, we get that any $2$-cocycle on $H$ with respect to the pair $(l,r)$ has the form \eqref{eq4.11}.

\item[2.5.] If $(\alpha_1,\beta_1)\neq (-1,0), (0,-1), (-1,-2), (-2,-1),$ then we obtain $$b_{1,3}=b_{2,3}=0, \quad \beta_1b_{1,4}=(1+\alpha_1)b_{1,5},\quad \alpha_1b_{2,5}=(1+\beta_1)b_{2,4}.$$

    In this case, we conclude that any $2$-cocycle on $H$ with respect to the pair $(l,r)$ has the form \eqref{eq4.12}.
\end{itemize}

\end{itemize}
\end{proof}

By Propositions \ref{B2} and \ref{Z2forH}, we obtain the following corollary:
\begin{cor}\label{cor4.4} We have the following:
\begin{itemize}
\item [I.] If $\alpha_1=\beta_1=\beta_2=0,$ $\alpha_2=1,$ then  $\operatorname{dim} Z^2(H,l,r) =6,$
$\operatorname{dim} B^2(H,l,r) =5;$
\item [II.] If $\alpha_1=\alpha_2=\beta_1=0,$ $\beta_2=1,$ then  $\operatorname{dim} Z^2(H,l,r) =6,$
$\operatorname{dim} B^2(H,l,r) =5;$

\item [III.] If $\alpha_1=\beta_1=\beta_2=0,$ $\alpha_2=2,$ then  $\operatorname{dim} Z^2(H,l,r) =6,$
$\operatorname{dim} B^2(H,l,r) =5;$
\item [IV.] If $\alpha_1=\alpha_2=\beta_1=0,$ $\beta_2=2,$ then  $\operatorname{dim} Z^2(H,l,r) =6,$
$\operatorname{dim} B^2(H,l,r) =5;$

\item [V.] If $\alpha_1=\beta_1=0,$ $\alpha_2=\beta_2=1$, then  $\operatorname{dim} Z^2(H,l,r) =6,$ $\operatorname{dim} B^2(H,l,r) =5;$

\item [VI.] If $\alpha_1=\beta_1=0,$ $ (\alpha_2,\beta_2)\neq (1,0),(0,1),(2,0),(0,2),(1,1),$ then $\operatorname{dim} Z^2(H,l,r) =5,$ $\operatorname{dim} B^2(H,l,r) =5;$

\item [VII.] If $\alpha_1=-\alpha_2=-1,$ $\beta_1=\beta_2=0$, then  $\operatorname{dim} Z^2(H,l,r) =5,$ $\operatorname{dim} B^2(H,l,r) =3;$

 \item [VIII.] If $\alpha_1=\alpha_2=0,$ $\beta_1=-\beta_2=-1$, then  $\operatorname{dim} Z^2(H,l,r) =5,$ $\operatorname{dim} B^2(H,l,r) =3;$

\item [IX.] If $\alpha_1=-\alpha_2=-1,$ $\beta_1=-\beta_2=-2$, then  $\operatorname{dim} Z^2(H,l,r) =5,$
$\operatorname{dim} B^2(H,l,r) =4;$
\item [X.] If $\alpha_1=-\alpha_2=-2,$ $\beta_1=-\beta_2=-1$, then  $\operatorname{dim} Z^2(H,l,r) =5,$
$\operatorname{dim} B^2(H,l,r) =4;$

\item [XI.] If $\alpha_1=-\alpha_2,$ $\beta_1=-\beta_2,$ and $(\alpha_1, \beta_1) \neq (-1,0), (0,-1), (-2,-1), (-1,-2),$ then $\operatorname{dim} Z^2(H,l,r) =4,$
$\operatorname{dim} B^2(H,l,r) =4.$

 \end{itemize}
\end{cor}

In the following theorem, we determine all one-dimensional abelian extensions of the solvable Leibniz algebra $H.$
  \begin{thm}\label{thmH} Let $\widehat{H}$ be an extension of the solvable Leibniz algebra $H$ by the abelian algebra $\mathfrak{h}=\langle e_4 \rangle.$ Then
   $\widehat{H}$ is isomorphic to one of the following non-isomorphic algebras:
$$\widehat{H}_1:\begin{array}{lllll} [e_1,e_2]=e_3,& [e_2,e_1]=-e_3,& [e_1,x_1]=e_1, & [x_1,e_1]=-e_1,\\ [1mm]
[e_3,x_1]=e_3,&[x_1,e_3]=-e_3,& [e_4,x_1]=e_4,& [x_2,e_1]=e_4,\\ [1mm]
[e_2,x_2]=e_2,& [x_2,e_2]=-e_2,& [e_3,x_2]=e_3, & [x_2, e_3]=-e_3.\end{array} $$
$$\widehat{H}_2:\begin{array}{lllll} [e_1,e_1]=e_4,& [e_1,e_2]=e_3,& [e_2,e_1]=-e_3,& [e_1,x_1]=e_1, \\ [1mm]
[x_1,e_1]=-e_1, & [e_3,x_1]=e_3,&[x_1,e_3]=-e_3,& [e_4,x_1]=2e_4,\\ [1mm]
[e_2,x_2]=e_2,& [x_2,e_2]=-e_2,& [e_3,x_2]=e_3, & [x_2, e_3]=-e_3.\end{array} $$
$$\widehat{H}_3:\begin{array}{lllll}  [e_1,e_2]=e_3+e_4,& [e_2,e_1]=-e_3+e_4,& [e_1,x_1]=e_1, & [x_1,e_1]=-e_1,\\ [1mm]
 [e_3,x_1]=e_3,&[x_1,e_3]=-e_3+e_4,& [e_4,x_1]=e_4,& [e_4,x_2]=e_4,\\ [1mm]
[e_2,x_2]=e_2,& [x_2,e_2]=-e_2,& [e_3,x_2]=e_3, & [x_2, e_3]=-e_3+e_4.\end{array} $$
$$\widehat{H}_4:\begin{array}{lllll} [e_1,e_2]=e_3,& [e_2,e_1]=-e_3,& [e_1,x_1]=e_1+e_4, & [x_1,e_1]=-e_1-e_4,\\ [1mm]
[e_3,x_1]=e_3,&[x_1,e_3]=-e_3,& [e_4,x_1]=e_4,& [x_1,e_4]=-e_4, \\ [1mm]
[e_2,x_2]=e_2,& [x_2,e_2]=-e_2,& [e_3,x_2]=e_3, & [x_2, e_3]=-e_3.\end{array} $$
$$\widehat{H}_5(\delta):\begin{array}{lllll} [e_1,e_2]=e_3,& [e_2,e_1]=-e_3,& [e_1,x_1]=e_1+\delta e_4, & [x_1,e_1]=-e_1-\delta e_4,\\ [1mm]
[e_1,x_2]=e_4, & [x_2,e_1]=-e_4, & [e_3,x_1]=e_3,&[x_1,e_3]=-e_3, \\ [1mm]
[e_4,x_1]=e_4,& [x_1,e_4]=-e_4, & [e_2,x_2]=e_2,& [x_2,e_2]=-e_2,\\[1mm]
 [e_3,x_2]=e_3, & [x_2, e_3]=-e_3.\end{array} $$
$$\widehat{H}_6:\begin{array}{lllll} [e_1,e_2]=e_3,& [e_2,e_1]=-e_3,&[e_1,e_3]=e_4, & [e_3,e_1]=-e_4,\\ [1mm]
 [e_1,x_1]=e_1, & [x_1,e_1]=-e_1, & [e_3,x_1]=e_3, &[x_1,e_3]=-e_3,\\ [1mm]
 [e_4,x_1]=2e_4, & [x_1,e_4]=-2e_4, &   [e_2,x_2]=e_2,& [x_2,e_2]=-e_2, \\ [1mm]
[e_3,x_2]=e_3, & [x_2, e_3]=-e_3, & [e_4,x_2]=e_4, &[x_2,e_4]=-e_4.\end{array} $$

\end{thm}
\begin{proof} By Corollary \ref{cor4.4}, we have that $H^2(H,l,r)=0$ in the cases VI and XI. Thus, it is enough to consider the remaining cases.

\begin{itemize}
\item [I.] Let $\alpha_1=\beta_1=\beta_2=0,$ $\alpha_2=1.$ In this case $\dim H^2(H,l,r) =1$ and $2$-cocycle defined by $\omega(x_2,e_1)=e_4$ provides a basis for this space, i.e.,  $H^2(H,l,r) =\langle\overline{\omega}\rangle.$

For any automorphism $\varphi_1 \in \operatorname{Aut}(H)$ as in Proposition \ref{Aut3-dim}, and any $\psi \in \operatorname{Aut}(\mathfrak{h}),$ where $\psi(e_{4}) = \lambda e_{4}$, the element $\delta\overline{\omega} \in H^2(H,l,r)$ acts on  $\overline{\omega'} \in H^2(H,l,r)$, such that $\omega'(x_2, e_1) = \delta \lambda a_{1,1}e_{4},$ with $l'=l,$ $r'=r.$ Taking $\lambda =1$ and  $a_{1,1} = \frac 1 {\delta},$ we can assume that $\omega'(x_2, e_1) = e_{4}.$

Moreover, if we consider the action by the automorphism $\varphi_2$ of Proposition \ref{Aut3-dim}, then the element $\overline{\omega}$ acts on the element $\overline{\omega'} \in H^2(H,l,r)$, such that $\omega'(x_1, e_2) =  \lambda e_{4},$ with $r'_{x_1}=r_{x_2},$ $r'_{x_2}=r_{x_1}.$
Note that the 2-cocycle $\omega'(x_1, e_2) =  \lambda e_{4}$ provides a basis for the case II.

Therefore, we conclude that all elements of $H^2(H,l,r)$ in the cases I and II are in one orbit. Thus, we can define the products of the extension algebra $\widehat{H}=H\oplus\langle e_4 \rangle$ as follows:
$$[x_2,e_1]=\omega(x_2,e_1)=e_4, \quad [e_4,x_1]=r_{x_1}e_{4}=e_4. $$

Hence, we obtain the algebra $\widehat{H}_1$.

\item [III.] Let $\alpha_1=\beta_1=\beta_2=0,$ $\alpha_2=2.$ Then $\overline{\omega}$ is a basis of $H^2(H,l,r),$ where $\omega(e_1,e_1)=e_4.$
Automorphisms $\varphi_1 \in \operatorname{Aut}(H)$ and $\psi \in \operatorname{Aut}(\mathfrak{h})$ transform  $\delta\overline{\omega}$ to $\delta'\overline{\omega},$ where $$\delta'=\delta \lambda a_{1,1}^2.$$

Taking $\lambda =\frac 1 {\delta}$ and  $a_{1,1} = 1,$ we can suppose that $\delta'= 1.$

Moreover, if we consider the action by the automorphism $\varphi_2$, then the element $\overline{\omega}$ acts on the element $\overline{\omega'}$, such that $\omega'(e_2, e_2) =  \lambda e_{4},$ with $r'_{x_1}=r_{x_2},$ $r'_{x_2}=r_{x_1}.$
Note that the cocycle $\omega'(e_2, e_2) =  \lambda e_{4},$ provides a basis for Case IV.
Therefore, we derive that all elements of $H^2(H,l,r)$ in the cases III and IV belong to the same orbit.
Thus, we can define the products of the extension algebra $\widehat{H}=H\oplus\langle e_4 \rangle$ as follows:
$$[e_1,e_1]=\omega(e_1,e_1)=e_4, \quad [e_4,x_1]=r_{x_1}e_{4}=2e_4. $$

Therefore, we obtain the algebra  $\widehat{H}_2$.

\item [V.] Let $\alpha_1=\beta_1=0,$ $\alpha_2=\beta_2=1.$ Then $\dim H^2(H,l,r) =1$ and $2$-cocycle defined as $$\omega(e_1,e_2)=\omega(e_2,e_1)=\omega(x_1,e_3)=\omega(x_2,e_3)=e_4,$$ provides a basis on $H^2(H,l,r).$ An automorphisms $\varphi_1 \in \operatorname{Aut}(H)$ and $\psi \in \operatorname{Aut}(\mathfrak{h})$ act the element $\delta\overline{\omega}$ to $\delta'\overline{\omega}$ as $$\delta'=\delta \lambda a_{1,1} a_{2,2}.$$

Taking $\lambda =\frac 1 {\delta}$ and  $a_{1,1} = a_{2,2}=1,$ we can suppose $\delta'= 1.$
Define the products of the extension algebra $\widehat{H}=H\oplus\langle e_4 \rangle$ as follows:
$$\begin{array}{lllll} [e_1,e_2]=e_3+\omega(e_1,e_2)=e_3+e_4,& [e_2,e_1]=-e_3+\omega(e_2,e_1)=-e_3+e_4,\\[1mm]
[x_1,e_3]=-e_3+\omega(x_1,e_3)=-e_3+e_4, & [x_2,e_3]=-e_3+\omega(x_2,e_3)=-e_3+e_4,\\[1mm] [e_4,x_1]=r_{x_1}e_{4}=e_4,  & [e_4,x_2]=r_{x_2}e_{4}=e_4. \end{array}$$

Therefore, we obtain the algebra $\widehat{H}_3$.

\item [VII.] $\alpha_1=-\alpha_2=-1,$ $\beta_1=\beta_2=0.$ In this case $\dim H^2(H,l,r) =2$ and the basis of this space defined by the 2-cocycles
    $$\begin{array}{lllll}\omega_1(e_1, x_1) =e_4, &\omega_1(x_1, e_1)=-e_{4}, && \omega_2(e_1, x_2)=e_4, &\omega_2(x_2, e_1) =-e_{4}.\end{array}$$

Any automorphism $\varphi_1 \in \operatorname{Aut}(H)$ and $\psi \in \operatorname{Aut}(\mathfrak{h})$ act the element $\delta_1\overline{\omega_1}+\delta_2\overline{\omega_2}$ to $\delta_1'\overline{\omega_1}+\delta_2'\overline{\omega_2}$ as
$$\delta_1'=\delta_1 a_{1,1},\quad \delta_2'=\delta_2a_{1,1}.$$

\begin{itemize}
\item If $\delta_2=0$, then
$\delta_2'=0$ and $\delta_1\neq 0.$ Taking $ a_{1,1} = \frac 1 {\delta_1},$ we can suppose $\delta_1'=1$ and obtain the algebra $\widehat{H}_4.$
\item If $\delta_2\neq 0$, then choosing $a_{1,1}=\frac{1}{\delta_2},$ we get $\delta_2'=1$
and obtain the algebra $\widehat{H}_5(\delta).$
\end{itemize}

Moreover, under the action of the automorphism $\varphi_2$ the element $\delta_1\overline{\omega_1}+\delta_2\overline{\omega_2}$ is mapped to the element $\delta_1\overline{\omega_1'}+\delta_2\overline{\omega_2'}$, where
$$\begin{array}{lllll}\omega_1'(e_2, x_1) =e_4, &\omega_1'(x_1, e_2)=-e_{4}, && \omega_2'(e_2, x_2)=e_4, &\omega_2'(x_2, e_2) =-e_{4},\end{array}$$
and $l'_{x_1}=l_{x_2},$ $l'_{x_2}=l_{x_1},$ $r'_{x_1}=r_{x_2},$ $r'_{x_2}=r_{x_1}.$
Note that the cocycles $\omega'_1$ and $\omega'_2$ form a basis for Case VIII. Therefore, Case VIII also yields the algebras $\widehat{H}_4$ and $\widehat{H}_5(\delta).$

\item [IX.] $\alpha_1=-\alpha_2=-2,$ $\beta_1=-\beta_2=-1.$ Then 2-cocycle $\omega(e_1, e_3) = e_{4},$ $\omega(e_3, e_1) =-e_{4}$ form a basis of $H^2(H,l,r)$ and
automorphisms $\varphi_1 \in \operatorname{Aut}(H)$ and $\psi \in \operatorname{Aut}(\mathfrak{h})$ act the element $\delta\overline{\omega}$ to $\delta'\overline{\omega}$ as $$\delta'=\delta \lambda a_{1,1}^2 a_{2,2}.$$

Taking $\lambda =\frac 1 {\delta}$ and  $a_{1,1} = a_{2,2}=1,$ we can suppose $\delta'= 1.$
The new products of the algebra  $\widehat{H}=H\oplus\langle e_4 \rangle$ are defined as follows:
    $$\begin{array}{lllll} [e_1,e_3]=\omega(e_1,e_3) =  e_{4}, & [e_3,e_1]=-\omega(e_1,e_3) =-e_{4},\\[1mm] [e_4,x_1]=r_{x_1}e_{4}=2e_4, & [x_1,e_4]=l_{x_1}e_{4}=-2e_4, \\[1mm]
[e_4,x_2]=r_{x_2}e_{4}=e_4, & [x_2,e_4]=l_{x_2}e_{4}=-e_4, \end{array}$$

Therefore, we get the algebra $\widehat{H}_6$.

Moreover, under the action of the automorphism $\varphi_2$, the element $\delta\overline{\omega}$ is mapped to the element $\delta\overline{\omega'}$, where
$$\begin{array}{lllll}\omega(e_2, e_3) = e_4, &\omega(e_3, e_2)=-e_{4},\end{array}$$
and $l'_{x_1}=l_{x_2},$ $l'_{x_2}=l_{x_1},$ $r'_{x_1}=r_{x_2},$ $r'_{x_2}=r_{x_1}.$
Note that the cocycle $\omega$ forms a basis in Case X. Therefore, Case X also yields the algebra $\widehat{H}_6.$
\end{itemize}
\end{proof}

\subsection{Extensions of the algebra $L_1$}
Now, using the algorithm for constructing of solvable Leibniz algebras, we obtain all extensions of the solvable Leibniz algebra $L_1$
by the one-dimensional abelian algebra $\mathfrak{h}=\langle e_4 \rangle$.
The condition \eqref{c3}, for the linear maps $l,r: L_1 \rightarrow \operatorname{End}(\mathfrak{h})$ such that
$$\begin{array}{ll}l_{x_1}(e_{4})=\alpha_1 e_{4}, & r_{x_1}(e_{4})=\alpha_2e_{4},\\[1mm]
l_{x_2}(e_{4})=\beta_1 e_{4}, & r_{x_2}(e_{4})=\beta_2e_{4}.\end{array}$$
leads to the following equalities:
\begin{equation}\label{L_1z}\begin{array}{ll} \alpha_1(\alpha_1+\alpha_2)=0, & \alpha_1(\beta_1+\beta_2)=0,\\[1mm]
\beta_1(\alpha_1+\alpha_2)=0,& \beta_1(\beta_1+\beta_2)=0.\end{array}\end{equation}

In the following Proposition, we describe all $2$-cocycles on $L_1$ with respect to the pair $(l,r),$ where $\mathfrak{h}=\langle e_4 \rangle$.

\begin{prop}{\label{Z2forL1}}
A basis of $Z^2(L_1,l,r)$ consists of the following cocycles:
\begin{itemize}
\item[I.] $\alpha_1=\beta_1=\beta_2=0,$ $\alpha_2=1:$
\begin{equation}\label{eq4.13}\begin{array}{llll}\omega(e_2,e_1)=b_{2,1}, & \omega(e_2,x_1)= b_{2,4} & \omega(e_2,x_2)=-b_{2,4}, &\omega(e_3,x_2)=b_{2,1}, \\ [1mm]
\omega(x_1,e_1)=b_{4,1},&\omega(x_1,x_1)= b_{4,4}, &  \omega(x_2,e_1)=b_{5,1},
& \omega(x_2,x_1)=b_{5,4}.\end{array}\end{equation}
\item[II.] $\alpha_1=\beta_1=\beta_2=0,$ $\alpha_2=2:$
\begin{equation}\label{eq4.14}\begin{array}{llll} \omega(e_1,e_1)=b_{1,1}, & \omega(e_1,x_1)=b_{1,4}, & \omega(e_2,e_1)=b_{2,1}, &  \omega(e_2,x_1)=-2b_{2,5},\\ [1mm]
\omega(e_2,x_2)= b_{2,5} &
\omega(e_3,x_1)=-b_{2,1}, & \omega(e_3,x_2)=b_{2,1},&
 \omega(x_1,e_1)=b_{1,4}, \\ [1mm]
\omega(x_1,x_1)= b_{4,4}, & \omega(x_2,x_1)=b_{5,4}.\end{array}\end{equation}
\item[III.] $\alpha_1=\beta_1=0,$ $\alpha_2=2,$ $\beta_2=1:$
\begin{equation}\label{eq4.15}\begin{array}{llll} \omega(e_1,x_1)=b_{1,4},& \omega(e_1,x_2)= b_{1,4}, &
\omega(e_2,e_1)=b_{2,1}, &\omega(e_2,x_1)=b_{2,4}, \\ [1mm]
\omega(e_3,e_1)=b_{3,1},& \omega(e_3,x_1)=-b_{2,1}, &
\omega(x_1,e_1)=b_{1,4}, & \omega(x_1,x_1)=2 b_{4,5}, \\ [1mm]
\omega(x_1,x_2)= b_{4,5}, & \omega(x_2,x_1)= 2b_{5,5}, & \omega(x_2,x_2)=b_{5,5}.\end{array}\end{equation}
\item[IV.] $\alpha_1=\alpha_2=\beta_1=0, $ $\beta_2=1:$
\begin{equation}\label{eq4.16}\begin{array}{lllll} \omega(e_1,x_1)=b_{1,4},& \omega(e_1,x_2)= -b_{1,4},& \omega(e_2,e_1)=b_{2,1},\\ [1mm]
\omega(e_2,x_1)=b_{2,4}, & \omega(e_2,x_2)=b_{2,5},& \omega(e_3,x_1)=b_{2,1},\\ [1mm]
\omega(x_1,e_1)=-b_{1,4}, &
\omega(x_1,x_2)= b_{4,5},  & \omega(x_2,x_2)=b_{5,5}.\end{array}\end{equation}
\item[V.] $\alpha_1=\beta_1=0, $  $(\alpha_2,\beta_2)\neq (1,0), (2,0), (2,1), (0,1):$
\begin{equation}\label{eq4.17}\begin{array}{lllll} \omega(e_1,x_1)=(\alpha_2-1)b_{4,1}, & \omega(e_1,x_2)= \beta_2b_{4,1},&
 \omega(e_2,e_1)=b_{2,1}, \\ [1mm]
 \omega(e_2,x_1)=\alpha_2b_{2,4}, & \omega(e_2,x_2)=(\beta_2-1)b_{2,4},& \omega(e_3,x_1)=(1-\alpha_2)b_{2,1}, \\ [1mm]
\omega(e_3,x_2)=(1-\beta_2)b_{2,1},&
 \omega(x_1,e_1)=b_{4,1}, &
 \omega(x_1,x_1)=\alpha_2 b_{4,5}, \\ [1mm]
  \omega(x_1,x_2)= \beta_2b_{4,5}, &\omega(x_2,x_1)= \alpha_2 b_{5,5}, & \omega(x_2,x_2)=\beta_2b_{5,5}.\end{array}\end{equation}
\item[VI.] $\alpha_1=-\alpha_2=-1,$ $\beta_1=\beta_2=0:$
 \begin{equation}\label{eq4.18}\begin{array}{llll} \omega(e_1,x_1)=b_{1,4},&  \omega(e_1,x_2)=b_{1,5},&\omega(e_2,e_1)=b_{2,1},& \omega(e_2,x_1)=-b_{2,5},\\ [1mm]
 \omega(e_2,x_2)=b_{2,5}, & \omega(e_3,x_2)=b_{2,1}, & \omega(x_1,e_1)=-b_{1,4}, & \omega(x_1,e_2)=b_{2,5},\\ [1mm]
 \omega(x_1,e_3)= b_{2,1},&  \omega(x_1,x_2)=b_{4,5},  & \omega(x_2,e_1)=-b_{1,5},
 & \omega(x_2,x_1)=-b_{4,5}.\end{array}\end{equation}
\item[VII.] $\alpha_1=-\alpha_2,$ $\beta_1=-\beta_2,$ $(\alpha_1, \beta_1) \neq (-1,0):$
 \begin{equation}\label{eq4.19}\begin{array}{llll}\omega(e_1,x_1)=(1+\alpha_1)b_{1,5}, & \omega(e_1,x_2)=\beta_1b_{1,5}, &
  \omega(e_2,e_1)=b_{2,1},  \\ [1mm]
   \omega(e_2,x_1)=\alpha_1b_{2,5},& \omega(e_2,x_2)=(1+\beta_1)b_{2,5},&
  \omega(e_3,x_1)=(1+\alpha_1)b_{2,1},  \\ [1mm]
  \omega(e_3,x_2)=(1+\beta_1)b_{2,1}, & \omega(x_1,e_1)=-(1+\alpha_1)b_{1,5},&\omega(x_1,e_2)=-\alpha_1b_{2,5},\\[1mm]
  \omega(x_1,e_3)=-\alpha_1 b_{2,1},&\omega(x_1,x_2)=b_{4,5},&
\omega(x_2,e_1)=-\beta_1b_{1,5},\\[1mm]
\omega(x_2,e_2)=-\beta_1b_{2,5},&
 \omega(x_2,e_3)=-\beta_1 b_{2,1},
   & \omega(x_2,x_1)=-b_{4,5}.\end{array}\end{equation}

\end{itemize}

\end{prop}
\begin{proof}
For any $\omega\in Z^2(L_1,l,r),$
denote by $\omega(x,y)=b_{i,j}e_4,$ $1\leq i,j\leq 5$, where $x,y\in \{e_1,e_2,e_3,x_1,x_2\}.$
Using the equality \eqref{c4}, we obtain the following relations:
$$\small{\begin{array}{lllll}
\beta_1b_{1,1}=0, & \beta_2b_{1,1}=0, & \alpha_1b_{1,1}=0, & (2-\alpha_2)b_{1,1}=0,\\[1mm]
\beta_1b_{1,2}=0,&  \beta_2b_{1,2}=0, & (1-\beta_2)b_{1,2}=0,& (1+\alpha_1)b_{1,2}=0,\\[1mm]
(1-\alpha_2)b_{1,2}=0,& b_{1,3}=0,&  b_{2,3}=0, & b_{3,2}=0,\quad b_{3,3}=0,\\[1mm]
\beta_1b_{1,4}=(\alpha_2-1)b_{5,1},&\beta_2b_{1,4}=(\alpha_2-1)b_{1,5},&
b_{1,4}=(\alpha_1+\alpha_2-1)b_{4,1},& (\alpha_1+1)b_{1,4}=(\alpha_2-1)b_{4,1},\\[1mm]
\beta_1b_{1,5}=\beta_2b_{5,1},& b_{1,5}=\alpha_1b_{5,1}+\beta_2b_{4,1},\\[1mm]
\alpha_1b_{2,2}=0, & \alpha_2b_{2,2}=0,& \beta_1b_{2,2}=0, & (2-\beta_2)b_{2,2}=0,\\[1mm]
(1-\beta_2)b_{2,2}=0,& \alpha_1b_{2,4}=\alpha_2b_{4,2},& \beta_1b_{2,4}=\alpha_2b_{5,2}, \\[1mm]
\alpha_1b_{2,5}=(\beta_2-1)b_{4,2},& \alpha_2b_{2,5}=(\beta_2-1)b_{2,4},&\beta_1b_{2,5}=(\beta_2-1)b_{5,2}\\[1mm]
\beta_1a_{3,1}=0, & (1-\beta_2)b_{3,1}=0,&\alpha_1a_{3,1}=0 &(2-\alpha_2)b_{3,1}=0,\\[1mm]
b_{3,4}=(1-\alpha_2)b_{2,1},& \beta_1b_{3,4}=(\alpha_2-1)b_{5,3},& (1-\beta_2)b_{3,4}=(1-\alpha_2)b_{3,5}, \\[1mm] b_{3,5}=(1-\beta_2)b_{2,1}, & \beta_1b_{3,5}=(\beta_2-1)b_{5,3},\\[1mm]
\beta_1b_{4,1}=(1-\alpha_2)b_{5,1},& \beta_1b_{4,2}=-\alpha_2b_{5,2}, & \beta_2b_{4,2}=-\alpha_1b_{5,2}, &(\alpha_1+\alpha_2)b_{4,2}=0, \\[1mm]
b_{4,3}=b_{1,2}-\alpha_1b_{2,1}, &\beta_1b_{4,3}=-\alpha_2b_{5,3}, &  \beta_2b_{4,3}=-\alpha_1b_{5,3},& (\alpha_1+\alpha_2)b_{4,3}=0, \\[1mm]
 \alpha_1b_{4,4}=0, & \beta_1b_{4,4}=0,&  \alpha_1b_{5,5}=0, & \beta_1b_{5,5}=0,  \\[1mm]
 (\alpha_1+\alpha_2)b_{4,5}=\beta_2b_{4,4},& (\beta_1+\beta_2)b_{5,1}=0, & (\beta_1+\beta_2)b_{5,2}=0,\\[1mm]
 b_{5,3}=-\beta_1b_{2,1},& (\beta_1+\beta_2)b_{5,3}=0,& \alpha_2b_{5,5}=(\beta_1+\beta_2)b_{5,4},\\[1mm]
\beta_2b_{4,4}=\alpha_2b_{4,5}-\alpha_1b_{5,4}, &\alpha_2b_{5,5}=\beta_2b_{5,4}-\beta_1b_{4,5}. \end{array}}$$

 By \eqref{L_1z}, we derive that $\alpha_1=\beta_1=0$ or $\alpha_1=-\alpha_2$ and $\beta_1=-\beta_2.$ Thus, we consider the following two cases:

\begin{itemize}
\item[Case 1.]Let $\alpha_1=\beta_1=0.$  Then we have
$$\begin{array}{llll}
b_{1,2}=0, & b_{1,3}=0,& b_{1,4}= (\alpha_2-1)b_{4,1}, & b_{1,5}= \beta_2b_{4,1}, \\[1mm]
b_{2,2}=0, &b_{2,3}=0, & b_{3,2}=0,& b_{3,3}=0, \\[1mm]
 b_{3,3}=0,& b_{3,4}= (1-\alpha_2)b_{2,1}, & b_{3,5}= (1-\beta_2)b_{2,1}, \\[1mm]
b_{4,2}=0, & b_{4,3}= 0, & b_{5,2}=0, & b_{5,3}= 0,
 \end{array}$$
and the following restrictions
$$\begin{array}{llll}
 \beta_2b_{1,1}=0, & (\alpha_2-2)b_{1,1}=0,& \alpha_2b_{2,5}=(\beta_2-1)b_{2,4},\\[1mm]
 (\beta_2-1)b_{3,1}=0, & (\alpha_2-2)b_{3,1}=0,&  \alpha_2b_{4,5}=\beta_2b_{4,4}, \\[1mm]
\beta_2b_{5,1}=0, & (\alpha_2-1)b_{5,1}=0,& \alpha_2b_{5,5}=\beta_2b_{5,4}.\\[1mm]
 \end{array}$$
Now, consider the following subcases:

\begin{itemize}
\item[1.1.] If $\alpha_2=1,$ $\beta_2=0,$ then we obtain $b_{2,5}=-b_{2,4}$ and
$b_{1,1}=b_{3,1}=b_{4,5}= b_{5,5}=0.$ Hence, we conclude that any $2$-cocycle on $L_1$ with respect to the pair $(l,r)$ has the form \eqref{eq4.13}.

  \item[1.2.] If $\alpha_2=2,$ $\beta_2=0,$ then we obtain $b_{2,4}=-2b_{2,5}$ and
$b_{3,1}=b_{4,5}= b_{5,1}= b_{5,5}=0.$ Hence, we deduce that any $2$-cocycle on $L_1$ with respect to the pair $(l,r)$ has the form \eqref{eq4.14}.

\item[1.3.] If $\alpha_2=2,$ $\beta_2=1,$ then we obtain $b_{4,4}=2b_{4,5},$ $b_{5,4}=2b_{5,5}$   and
$b_{1,1}=b_{2,5}= b_{5,1}=0.$ Hence, we get that any $2$-cocycle on $L_1$ with respect to the pair $(l,r)$ has the form \eqref{eq4.15}.

\item[1.4.] If $\alpha_2=0,$ $\beta_2=1,$ then we obtain
$b_{1,1}=b_{3,1}=b_{4,4}=b_{5,1}=b_{5,4}=0.$ Hence, we get that any $2$-cocycle on $L_1$ with respect to the pair $(l,r)$ has the form \eqref{eq4.16}.

 \item[1.5.] If $ (\alpha_2,\beta_2)\neq (1,0),(2,0),(2,1),(0,1),$ then we obtain $b_{1,1}=b_{3,1}= b_{5,1}=0$ and $$ \alpha_2b_{2,5}=(\beta_2-1)b_{2,4},\quad \alpha_2b_{4,5}=\beta_2 b_{4,4},\quad \alpha_2b_{5,5}=\beta_2b_{5,4}.$$

 In this case, we obtain that any $2$-cocycle on $L_1$ with respect to the pair $(l,r)$ has the form \eqref{eq4.17}.
\end{itemize}

 \item [Case 2.] Let $\alpha_1=-\alpha_2$ and $\beta_1=-\beta_2.$  Then we have
$$\begin{array}{lllll}
b_{1,1}=0, & b_{1,2}=0,& b_{1,3}=0, & b_{2,2}=0, \\[1mm]
b_{3,1}=0, & b_{3,2}=0, & b_{3,3}=0,  & b_{3,4}= (1+\alpha_1)b_{2,1}, \\[1mm]
 b_{3,5}=(1+\beta_1)b_{2,1}, &   b_{4,1}=-b_{1,4},  & b_{4,2}=-b_{2,4},&
b_{4,3}= -\alpha_1b_{2,1}, \\[1mm] b_{4,4}=0,&
b_{5,1}=-b_{1,5},& b_{5,3}=- \beta_1b_{2,1},& b_{5,4}=-b_{4,5}, & b_{5,5}=0,
 \end{array}$$
and the following restrictions
$$
 \beta_1b_{1,4}=(1+\alpha_1)b_{1,5}, \quad \alpha_1b_{2,5}=(1+\beta_1)b_{2,4},\quad  \beta_1b_{2,5}=-(1+\beta_1)b_{5,2}.
$$

Consider the following subcases:
\begin{itemize}
\item[2.1.] If $\alpha_1=-1,$ $\beta_1=0,$ then we obtain $b_{2,5}=-b_{2,4}$ and
$b_{5,2}=0.$ Hence, we get that any $2$-cocycle on $L_1$ with respect to the pair $(l,r)$ has the form \eqref{eq4.18}.

\item[2.2.] If  $(\alpha_1, \beta_1) \neq (-1,0),$ then we conclude that any $2$-cocycle on $L_1$ with respect to the pair $(l,r)$ has the form \eqref{eq4.19}.

\end{itemize}
\end{itemize}

\end{proof}

By Propositions \ref{B2} and \ref{Z2forL1}, we obtain the following corollary.
\begin{cor}\label{cor4.7} We have the following:
\begin{itemize}
\item[I.] If $\alpha_1=\beta_1=\beta_2=0,$ $\alpha_2=1,$ then  $\operatorname{dim} Z^2(L_1,l,r) =6,$
$\operatorname{dim} B^2(L_1,l,r) =5;$
\item[II.] If $\alpha_1=\beta_1=\beta_2=0,$ $\alpha_2=2,$ then  $\operatorname{dim} Z^2(L_1,l,r) =6,$
$\operatorname{dim} B^2(L_1,l,r) =5;$
\item[III.] If $\alpha_1=\beta_1=0,$ $\alpha_2=2,$ $\beta_2=1$, then  $\operatorname{dim} Z^2(L_1,l,r) =6,$
$\operatorname{dim} B^2(L_1,l,r) =5;$
\item[IV.] If $\alpha_1=\alpha_2=\beta_1=0,$ $\beta_2=1$, then  $\operatorname{dim} Z^2(L_1,l,r) =6,$
$\operatorname{dim} B^2(L_1,l,r) =4;$
\item[V.] If $\alpha_1=\beta_1=0, $  $(\alpha_2,\beta_2)\neq (1,0), (2,0), (2,1), (0,1),$ then  $\operatorname{dim} Z^2(L_1,l,r) =5,$
$\operatorname{dim} B^2(L_1,l,r) =5;$
\item[VI.] If $\alpha_1=-\alpha_2=-1,$ $\beta_1=\beta_2=0$, then  $\operatorname{dim} Z^2(L_1,l,r) =5,$
$\operatorname{dim} B^2(L_1,l,r) =3;$
\item[VII.] If $\alpha_1=-\alpha_2,$ $\beta_1=-\beta_2,$ $(\alpha_1, \beta_1) \neq (-1,0),$ then  $\operatorname{dim} Z^2(L_1,l,r) =4,$
$\operatorname{dim} B^2(L_1,l,r) =4.$

 \end{itemize}
\end{cor}

Now, we can state the following result.
 \begin{thm}\label{thmL1} Let $\widehat{L}_1$ be an extension of the solvable Leibniz algebra $L_1$ by the abelian algebra $\mathfrak{h}=\langle e_4 \rangle.$ Then
   $\widehat{L}_1$ is isomorphic to one of the following pairwise non-isomorphic algebras:
  $$\widehat{L}^1_1:\begin{array}{lllll} [e_2,e_1]=e_3,& [e_1,x_1]=e_1, &[x_1,e_1]=-e_1, & [x_2,e_1]=e_4,\\ [1mm]
  [e_3,x_1]=e_3, & [e_2,x_2]=e_2,& [e_3,x_2]=e_3, & [e_4,x_1]=e_4,.\end{array} $$
$$\widehat{L}^2_1:\begin{array}{lllll} [e_1,e_1]=e_4,& [e_2,e_1]=e_3,& [e_1,x_1]=e_1, &[x_1,e_1]=-e_1, \\ [1mm]
  [e_3,x_1]=e_3, & [e_2,x_2]=e_2,& [e_3,x_2]=e_3, & [e_4,x_1]=2e_4.\end{array} $$
$$\widehat{L}^3_1:\begin{array}{llllll} [e_2,e_1]=e_3,& [e_3,e_1]=e_4,& [e_1,x_1]=e_1, &  [x_1,e_1]=-e_1, \\ [1mm]
 [e_3,x_1]=e_3, & [e_2,x_2]=e_2, & [e_3,x_2]=e_3, & [e_4,x_1]=2e_4,& [e_4,x_2]=e_4.\end{array} $$
 $$\widehat{L}^4_1:\begin{array}{lllll} [e_2,e_1]=e_3,& [e_1,x_1]=e_1, &[x_1,e_1]=-e_1,&   [e_2,x_1]=e_4, \\ [1mm]
[e_3,x_1]=e_3,& [e_2,x_2]=e_2,& [e_3,x_2]=e_3, & [e_4,x_2]=e_4.\end{array} $$
 $$\widehat{L}^5_1(\delta):\begin{array}{lllll} [e_2,e_1]=e_3,& [e_1,x_1]=e_1, &[x_1,e_1]=-e_1,&   [e_2,x_1]=\delta e_4, \\ [1mm]
[e_3,x_1]=e_3,& [e_2,x_2]=e_2+e_4,& [e_3,x_2]=e_3, & [e_4,x_2]=e_4.\end{array} $$
$$\widehat{L}^6_1:\begin{array}{lllll} [e_2,e_1]=e_3,& [e_1,x_1]=e_1+e_4, &[x_1,e_1]=-e_1-e_4,& [e_2,x_2]=e_2,\\ [1mm]
[e_3,x_1]=e_3,& [e_3,x_2]=e_3, & [e_4,x_1]=e_4, & [x_1,e_4]=-e_4,\end{array} $$
 $$\widehat{L}^7_1(\delta):\begin{array}{lllll} [e_2,e_1]=e_3,& [e_1,x_1]=e_1+\delta e_4, &  [e_1,x_2]=e_4,& [e_2,x_2]=e_2,\\ [1mm]
[e_3,x_1]=e_3,& [e_3,x_2]=e_3, &
[x_1,e_1]=-e_1-\delta e_4,&  [x_2,e_1]=-e_4,\\ [1mm]
[e_4,x_1]=e_4, & [x_1,e_4]=-e_4.\end{array} $$

 \end{thm}

\begin{proof} By Corollary \ref{cor4.7}, we have that $H^2(L_1,l,r)=0$ in cases V and VII. Thus, it is enough to consider the remaining  cases.

\begin{itemize}

\item [I.] $\alpha_1=\beta_1=\beta_2=0,$ $\alpha_2=1.$ In this case $\dim H^2(L_1,l,r) =1$ and $2$-cocycle defined by $\omega(x_2,e_1)=e_4$ form a basis of this space, i.e.,  $H^2(L_1,l,r) =\langle\overline{\omega}\rangle.$

Any automorphism $\varphi\in \operatorname{Aut}(L_1)$ from Proposition \ref{Aut3-dim}  and any automorphism $\psi \in \operatorname{Aut}(\mathfrak{h})$ with $\psi(e_{4}) = \lambda e_{4}$ act the element $\delta\overline{\omega} \in H^2(L_1,l,r)$ to $\delta'\overline{\omega} \in H^2(L_1,l,r)$, where
$$\delta' = \delta\lambda a_{1,1}, \quad l'=l, \quad r'=r.$$

By taking $\lambda =1$ and  $a_{1,1} = \frac 1 {\delta},$ we can assume  $\delta' =1.$
Thus, the new products of the algebra $\widehat{L}_1=L_1\oplus\langle e_4 \rangle$ are given by
    $$\begin{array}{llll}[x_2,e_1]=\omega(x_2,e_1) =e_{4},&
[e_4,x_1]=r_{x_1}e_{4}=e_4. \end{array}$$

Hence, we obtain  the algebra $\widehat{L}^1_1$.

\item [II.] $\alpha_1=\beta_1=\beta_2=0,$ $\alpha_2=2.$
In this case $\dim H^2(L_1,l,r) =1$ and $2$-cocycle defined by $\omega(e_1,e_1)=e_4$ forms a basis of this space, i.e.  $H^2(L_1,l,r) =\langle\overline{\omega}\rangle.$
Any automorphisms $\varphi \in \operatorname{Aut}(L_1)$ and $\psi \in \operatorname{Aut}(\mathfrak{h})$ act the element $\delta\overline{\omega}$ to $\delta'\overline{\omega}$, where
$$\delta' = \delta\lambda a_{1,1}^2.$$

Taking $\lambda =\frac 1 {\delta}$ and  $a_{1,1} = 1,$ we can suppose $\delta' =1.$
Thus, the new products of the algebra  $\widehat{L}_1=L_1\oplus\langle e_4 \rangle$ are defined as follows:
$$[e_1,e_1]=\omega(e_1,e_1)=e_4, \quad [e_4,x_1]=r_{x_1}e_{4}=2e_4. $$

Therefore, we obtain the algebra $\widehat{L}^2_1$.

\item [III.] $\alpha_1=\beta_1=0,$ $\alpha_2=2,$ $\beta_2=1.$
In this case $\dim H^2(L_1,l,r) =1$ and $H^2(L_1,l,r) =\langle\overline{\omega}\rangle,$ where $\omega(e_3,e_1)=e_4.$
Automorphisms $\alpha \in \operatorname{Aut}(L_1)$ and  $\psi \in \operatorname{Aut}(\mathfrak{h})$ act the element $\delta\overline{\omega}$ to $\delta'\overline{\omega}$, as follows:
$$\delta' = \delta\lambda a_{1,1}^2a_{2,2}.$$

Taking $\lambda =\frac 1 {\delta}$ and  $a_{1,1} = a_{2,2}=1,$ we can suppose $\delta' =1.$
Thus, the new products of the algebra  $\widehat{L}_1=L_1\oplus\langle e_4 \rangle$ are given by:
    $$\begin{array}{llll}[e_3,e_1]=\omega(e_3,e_1) =e_{4},&
[e_4,x_1]=r_{x_1}e_{4}=2e_4, & [e_4, x_2]=r_{x_2}e_{4}=e_4. \end{array}$$

Therefore, we obtain  the algebra $\widehat{L}^3_1$.

\item [IV.] $\alpha_1=\alpha_2=\beta_1=0, $ $\beta_2=1.$
In this case $\dim H^2(L_1,l,r) =2$ and the basis of this space is given by the 2-cocycles
    $$\begin{array}{lllll}\omega_1(e_2, x_1) =e_4, && \omega_2(e_2, x_2)=e_4.\end{array}$$

Automorphisms $\varphi\in \operatorname{Aut}(L_1)$ and $\psi \in \operatorname{Aut}(\mathfrak{h})$ act the element $\delta_1\overline{\omega_1}+\delta_2\overline{\omega_2}$ to $\delta_1'\overline{\omega_1}+\delta_2'\overline{\omega_2}$ as
$$\delta_1'=\delta_1 \lambda a_{2,2},\quad \delta_2'=\delta_2\lambda a_{2,2}.$$

\begin{itemize}
\item If $\delta_2=0$, then
$\delta_2'=0$ and $\delta_1\neq 0.$ Taking $\lambda=1,$ $a_{2,2} = \frac 1 {\delta_1},$ we can suppose $\delta_1'=1$ and obtain the algebra $\widehat{L}_1^4.$
\item If $\delta_2\neq 0$, then choosing $\lambda=1,$ $a_{2,2}=\frac{1}{\delta_2},$ we get $\delta_2'=1$
and obtain the algebra $\widehat{L}_1^5(\delta).$
\end{itemize}

\item [VI.] $\alpha_1=-\alpha_2=-1,$ $\beta_1=\beta_2=0.$
In this case $\dim H^2(L_1,l,r) =2$ and the basis of this space is given by the 2-cocycles
    $$\begin{array}{lllll}\omega_1(e_1, x_1) =e_4, &\omega_1(x_1, e_1)=-e_{4}, && \omega_2(e_1, x_2)=e_4, &\omega_2(x_2, e_1) =-e_{4}.\end{array}$$

Automorphisms $\varphi \in \operatorname{Aut}(L_1)$ and $\psi \in \operatorname{Aut}(\mathfrak{h})$ act the element $\delta_1\overline{\omega_1}+\delta_2\overline{\omega_2}$ to $\delta_1'\overline{\omega_1}+\delta_2'\overline{\omega_2}$ as
$$\delta_1'=\delta_1 \lambda a_{1,1},\quad \delta_2'=\delta_2\lambda a_{1,1}.$$

\begin{itemize}
\item If $\delta_2=0$ , then
$\delta_2'=0$ and $\delta_1\neq 0.$ Taking $\lambda=1,$ $a_{1,1} = \frac 1 {\delta_1},$ we can suppose $\delta_1'=1$ and obtain the algebra $\widehat{L}_1^6.$
\item If $\delta_2\neq 0$, then choosing $\lambda=1,$ $a_{1,1}=\frac{1}{\delta_2},$ we get $\delta_2'=1$ and obtain the algebra $\widehat{L}_1^7(\delta).$
\end{itemize}
\end{itemize}

\end{proof}

\subsection{Extensions of the algebra $L_2$}

Similarly to the previous subsections, we obtain the one-dimensional abelian extensions of the algebra $L_2$ as follows.

\begin{prop}{\label{L2}}
A basis of $Z^2(L_2,l,r)$ consists of the following cocycles:
\begin{itemize}
\item $\alpha_1=\beta_1=\beta_2=0,$ $\alpha_2=1:$
$$\begin{array}{llll}\omega(e_1,e_1)=b_{1,1}, &\omega(e_2,x_1)=-b_{2,5},& \omega(e_2,x_2)=b_{2,5}, &\omega(e_3,x_1)=b_{1,1}, \\ [1mm]
\omega(x_1,e_1)=b_{4,1}, & \omega(x_1,x_1)=b_{4,4}, &  \omega(x_2,e_1)=b_{5,1}, & \omega(x_2,x_1)=b_{5,4}. \\ [1mm]
\end{array}$$
\item $\alpha_1=\beta_1=\beta_2=0,$ $\alpha_2=3:$
$$\begin{array}{lllll}\omega(e_1,e_1)=b_{1,1}, &\omega(e_1,x_1)=2b_{4,1},& \omega(e_2,x_1)=-3b_{2,5},
&\omega(e_2,x_2)=b_{2,5},& \\ [1mm]
\omega(e_3,e_1)=b_{3,1},&\omega(e_3,x_1)=-b_{1,1}, & \omega(x_1,e_1)=b_{4,1}, & \omega(x_1,x_1)=b_{4,4},\\ [1mm] \omega(x_2,x_1)=b_{5,4}.
\end{array}$$
\item $\alpha_1=\alpha_2=\beta_1=0, $ $\beta_2=1:$
$$\begin{array}{lllll}\omega(e_1,e_1)=b_{1,1},
& \omega(e_1,x_1)=b_{1,4},&\omega(e_1,x_2)=-b_{1,4},& \omega(e_2,x_1)=b_{2,4}, \\ [1mm]
\omega(e_2,x_2)=b_{2,5},&\omega(e_3,x_1)=2b_{1,1},&\omega(e_3,x_2)=-b_{1,1}, &
\omega(x_1,e_1)=-b_{1,4}, \\ [1mm]
 \omega(x_1,x_2)=b_{4,5}, & \omega(x_2,x_2)=b_{5,5}.
\end{array}$$
\item $\alpha_1= \beta_1=0, $  $\alpha_2=1,$ $\beta_2=1:$
$$\begin{array}{lllll}\omega(e_1,e_1)=b_{1,1}, & \omega(e_1,x_2)=b_{1,5},& \omega(e_2,e_1)=b_{2,1},
&\omega(e_2,x_1)=b_{2,4},\\ [1mm]
\omega(e_3,x_1)=b_{1,1}, &\omega(e_3,x_2)=-b_{1,1},
&\omega(x_1,e_1)=b_{1,5}, & \omega(x_1,x_1)=b_{4,4},  \\ [1mm]
\omega(x_1,x_2)=b_{4,4}, & \omega(x_2,x_1)=b_{5,5},& \omega(x_2,x_2)=b_{5,5}.
\end{array}$$
\item $\alpha_1=\beta_1=0, $ $(\alpha_2,\beta_2)\neq (1,0), (3,0), (0,1), (1,1):$
$$\begin{array}{lllll}\omega(e_1,e_1)=b_{1,1}, & \omega(e_1,x_1)=(\alpha_2-1)b_{4,1},& \omega(e_1,x_2)=\beta_2b_{4,1}, \\ [1mm] \omega(e_2,x_1)=\alpha_2b_{2,4},  &
 \omega(e_2,x_2)=(\beta_2-1)b_{2,4},&
 \omega(e_3,x_1)=(2-\alpha_2)b_{1,1},\\ [1mm]
 \omega(e_3,x_2)=-\beta_2b_{1,1}, & \omega(x_1,e_1)=b_{4,1}, &
\omega(x_1,x_1)=\alpha_2b_{4,4}, \\ [1mm]
\omega(x_1,x_2)=\beta_2b_{4,4},&
\omega(x_2,x_1)=\alpha_2b_{5,4}, & \omega(x_2,x_1)=\beta_2b_{5,4}.
\end{array}$$
\item $\alpha_1=-\alpha_2=-1,$ $\beta_1=\beta_2= 0:$
 $$\begin{array}{llllll} \omega(e_1,e_1)=b_{1,1},& \omega(e_1,x_1)=b_{1,4},& \omega(e_1,x_2)=b_{1,5},& \omega(e_2,x_1)=b_{2,4},\\[1mm]
 \omega(e_2,x_2)=-b_{2,4}, & \omega(e_3,x_1)=b_{1,1}, &
 \omega(x_1,e_1)=-b_{1,4},& \omega(x_1,e_2)=-b_{2,4}, \\ [1mm]
\omega(x_1,e_3)=b_{1,1},& \omega(x_1,x_2)=b_{4,5}, &
\omega(x_2,e_1)=-b_{1,5},
 & \omega(x_2,x_1)=-b_{4,5}.\end{array}$$
\item $\alpha_1=-\alpha_2,$ $\beta_1=-\beta_2,$ $(\alpha_1,\beta_1)\neq (-1,0):$
$$\begin{array}{llllll} \omega(e_1,e_1)=b_{1,1}, &
\omega(e_1,x_1)=-(1+\alpha_1)b_{5,1},& \omega(e_1,x_2)=-\beta_1b_{5,1}, \\ [1mm]
\omega(e_2,x_1)=-\alpha_1b_{5,2},&
\omega(e_2,x_2)=-(1+\beta_1)b_{5,2},&
\omega(e_3,x_1)=(2+\alpha_1)b_{1,1},\\ [1mm] \omega(e_3,x_2)=\beta_1b_{1,1}, &\omega(x_1,e_1)=(1+\alpha_1)b_{5,1},&
\omega(x_1,e_2)=\alpha_1b_{5,2}, \\ [1mm]
\omega(x_1,e_3)=-\alpha_1 b_{1,1},& \omega(x_1,x_2)=b_{4,5},& \omega(x_2,e_1)=\beta_1b_{5,1},\\ [1mm]
\omega(x_2,e_2)=\beta_1b_{5,2}, &
 \omega(x_2,e_3)=-\beta_1 b_{1,1},&  \omega(x_2,x_1)=-b_{4,5}.\end{array}$$
\end{itemize}

\begin{proof} The proof is similar to that of Propositions \ref{Z2forH} and \ref{Z2forL1}. \end{proof}

\end{prop}
\begin{cor}\label{corL.2} We have the following:
\begin{itemize}
\item If $\alpha_1=\beta_1=\beta_2=0$, $\alpha_2=1,$ then  $\operatorname{dim} Z^2(L_2,l,r) =6,$
$\operatorname{dim} B^2(L_2,l,r) =5;$
\item If $\alpha_1=\beta_1=\beta_2=0$, $\alpha_2=3,$  then  $\operatorname{dim} Z^2(L_2,l,r) =6,$
$\operatorname{dim} B^2(L_2,l,r) =5;$
\item If $\alpha_1=\alpha_2=\beta_1=0,$ $\beta_2=1$, then  $\operatorname{dim} Z^2(L_2,l,r) =6,$
$\operatorname{dim} B^2(L_2,l,r) =4;$
\item If $\alpha_1=\beta_1=0,$  $\alpha_2=1,$ $\beta_2=1$, then  $\operatorname{dim} Z^2(L_2,l,r) =6,$
$\operatorname{dim} B^2(L_2,l,r) =5;$
\item If $\alpha_1=\beta_1=0,$  $(\alpha_2,\beta_2)\neq (1,0), (3,0), (0,1), (1,1)$, then   $\operatorname{dim} Z^2(L_2,l,r) =5,$
$\operatorname{dim} B^2(L_2,l,r) =5;$
\item If $\alpha_1=-\alpha_2=-1,$ $\beta_1=\beta_2=0$, then  $\operatorname{dim} Z^2(L_2,l,r) =5,$
$\operatorname{dim} B^2(L_2,l,r) =3;$
\item If $\alpha_1=-\alpha_2,$ $\beta_1=-\beta_2,$ $(\alpha_1,\beta_1)\neq (-1,0),$ then $\operatorname{dim} Z^2(L_2,l,r) =4,$
$\operatorname{dim} B^2(L_2,l,r) =4;$

 \end{itemize}
\end{cor}

\begin{thm}
Let $\widehat{L}_2$ be an extension of the solvable Leibniz algebra $L_2$ by the abelian algebra $\mathfrak{h}=\langle e_4 \rangle.$ Then
   $\widehat{L}_2$ is isomorphic to one of the following non-isomorphic algebras:
  $$\widehat{L}^1_2:\begin{array}{lllll} [e_1,e_1]=e_3,& [e_1,x_1]=e_1, & [x_1,e_1]=-e_1, & [e_3,x_1]=2e_3,\\ [1mm] [x_2,e_1]=e_4, & [e_2,x_2]=e_2, & [e_4,x_1]=e_4.\end{array} $$
   $$\widehat{L}^2_2:\begin{array}{lllll} [e_1,e_1]=e_3,& [e_3,e_1]=e_4,& [e_1,x_1]=e_1, & [x_1,e_1]=-e_1, \\ [1mm] [e_3,x_1]=2e_3, & [e_2,x_2]=e_2, & [e_4,x_1]=3e_4.\end{array} $$
   $$\widehat{L}^3_2:\begin{array}{lllll} [e_1,e_1]=e_3,& [e_1,x_1]=e_1, &[x_1,e_1]=-e_1,& [e_2,x_1]=e_4, \\ [1mm]
  [e_3,x_1]=2e_3,&  [e_2,x_2]=e_2, & [e_4,x_2]=e_4.\end{array} $$
 $$\widehat{L}^4_2(\delta):\begin{array}{lllll} [e_1,e_1]=e_3,& [e_1,x_1]=e_1, &[x_1,e_1]=-e_1,& [e_2,x_1]=\delta e_4,  \\ [1mm]
  [e_3,x_1]=2e_3, & [e_1,x_2]=e_4,& [e_2,x_2]=e_2+e_4, & [e_4,x_2]=e_4.\end{array} $$
   $$\widehat{L}^5_2:\begin{array}{lllll} [e_1,e_1]=e_3,& [e_2,e_1]=e_4,& [e_1,x_1]=e_1, &[x_1,e_1]=-e_1, \\ [1mm]
  [e_3,x_1]=2e_3, & [e_2,x_2]=e_2,& [e_4,x_1]=e_4 & [e_4,x_2]=e_4.\end{array} $$
 $$\widehat{L}^6_2:\begin{array}{lllll} [e_1,e_1]=e_3,& [e_1,x_1]=e_1+e_4, &[x_1,e_1]=-e_1-e_4, \\ [1mm]
  [e_3,x_1]=2e_3,&  [e_2,x_2]=e_2, & [e_4,x_1]=e_4,& [x_1,e_4]=-e_4.\end{array} $$
 $$\widehat{L}^7_2(\delta):\begin{array}{lllll} [e_1,e_1]=e_3,& [e_1,x_1]=e_1+\delta e_4, &[x_1,e_1]=-e_1-\delta e_4, \\ [1mm]
  [e_3,x_1]=2e_3, & [e_1,x_2]=e_4,& [x_2,e_1]=-e_4,\\ [1mm] [e_2,x_2]=e_2, & [e_4,x_1]=e_4,& [x_1,e_4]=-e_4.\end{array} $$
 \end{thm}

\begin{proof}
The proof is similar to that of Theorems \ref{thmH} and \ref{thmL1}. \end{proof}

\subsection{Extensions of the algebra $L_3$}
Similarly to the previous subsections, we obtain one-dimensional abelian extensions of the algebra $L_3$ as follows.

 \begin{prop}{\label{L3}}
A basis of $Z^2(L_3,l,r)$ consists of the following cocycles:
\begin{itemize}
\item[1.] $\alpha_1=\beta_1=\beta_2=0,$ $\alpha_2=1:$
$$\begin{array}{lllll}\omega(e_1,e_1)=b_{1,1}, & \omega(e_2,x_1)=b_{5,2} ,& \omega(e_2,x_2)=-b_{5,2} , \\ [1mm]
\omega(e_3,x_1)=b_{1,1}, &
\omega(x_1,e_1)=b_{4,1}, & \omega(x_1,x_1)=b_{4,4}, \\ [1mm]
 \omega(x_2,e_1)=b_{5,1}, &\omega(x_2,e_2)=b_{5,2},& \omega(x_2,x_1)=b_{5,4}. \\ [1mm]
\end{array}$$

\item[2.] $\alpha_1=\beta_1=\beta_2=0,$ $\alpha_2=3:$
$$\begin{array}{lllll}\omega(e_1,e_1)=b_{1,1}, & \omega(e_1,x_1)=2b_{4,1},& \omega(e_2,x_1)=3b_{5,2},& \omega(e_2,x_2)=-b_{5,2}, \\ [1mm] \omega(e_3,e_1)=b_{3,1},&
 \omega(e_3,x_1)=-b_{1,1}, & \omega(x_1,e_1)=b_{4,1}, & \omega(x_1,x_1)=b_{4,4}, \\ [1mm]
 \omega(x_2,e_2)=b_{5,2},& \omega(x_2,x_1)=b_{5,4}. \\ [1mm]
\end{array}$$
\item[3.] $\alpha_1=\alpha_2=\beta_1=0, $ $\beta_2=1:$
$$\begin{array}{lllll}\omega(e_1,e_1)=b_{1,1}, & \omega(e_1,x_1)=-b_{4,1},& \omega(e_1,x_2)=b_{4,1},
& \omega(e_3,x_1)=2b_{1,1}, \\ [1mm]
\omega(e_3,x_2)=-b_{1,1}, &
\omega(x_1,e_1)=b_{4,1}, & \omega(x_1,e_2)=b_{4,2}, & \omega(x_1,x_2)=b_{4,5}, \\ [1mm]
 \omega(x_2,e_2)=b_{5,2},&
  \omega(x_2,x_2)=b_{5,5}. \\ [1mm]
\end{array}$$
\item[4.] $\alpha_1=\alpha_2=\beta_1=0, $ $\beta_2=2:$
$$\begin{array}{lllll}\omega(e_1,e_1)=b_{1,1},  & \omega(e_1,x_1)=-b_{4,1},& \omega(e_1,x_2)=2b_{4,1}, &\omega(e_2,e_2)=b_{2,2}, \\ [1mm]
\omega(e_2,x_2)=b_{5,2},& \omega(e_3,x_1)=2b_{1,1},&\omega(e_3,x_2)=-2b_{1,1},  & \omega(x_1,e_1)=b_{4,1},\\[1mm]
\omega(x_1,x_2)=b_{4,5}, &\omega(x_2,e_2)=b_{5,2},&
 \omega(x_2,x_2)=b_{5,5}. \\ [1mm]
\end{array}$$
\item[5.] $\alpha_1=\beta_1=0, $ $(\alpha_2,\beta_2)\neq (1,0),(3,0),(0,1),(0,2):$
$$\begin{array}{lllll}\omega(e_1,e_1)=b_{1,1}, & \omega(e_1,x_1)=(\alpha_2-1)b_{4,1},& \omega(e_1,x_2)=\beta_2b_{4,1},\\ [1mm] \omega(e_2,x_1)=\alpha_2b_{5,2},&
 \omega(e_2,x_2)=(\beta_2-1)b_{5,2},&\omega(e_3,x_1)=(2-\alpha_2)b_{1,1}, \\ [1mm]
 \omega(e_3,x_2)=-\beta_2b_{1,1},&  \omega(x_1,e_1)=b_{4,1}, &
\omega(x_1,x_1)=\alpha_2b_{4,5},\\ [1mm]
\omega(x_1,x_2)=\beta_2 b_{4,5},&\omega(x_2,e_2)=b_{5,2},& \omega(x_2,x_1)=\alpha_2b_{5,5},\\ [1mm]
\omega(x_2,x_2)=\beta_2b_{5,5}.
\end{array}$$
\item[6.] $\alpha_1=-\alpha_2=-1,$ $\beta_1=\beta_2= 0:$
 $$\begin{array}{llllll} \omega(e_1,e_1)=b_{1,1}, &\omega(e_1,x_1)=-b_{4,1},& \omega(e_1,x_2)=-b_{5,1},&\omega(e_2,x_1)=b_{5,2},\\ [1mm] \omega(e_2,x_2)=-b_{5,2},&
 \omega(e_3,x_1)=b_{1,1}, &\omega(x_1,e_1)=b_{4,1},& \omega(x_1,e_2)=-b_{5,2},\\ [1mm]
 \omega(x_1,e_3)=b_{1,1},  &\omega(x_1,x_2)=b_{4,5},&
\omega(x_2,e_1)=b_{5,1}, &
\omega(x_2,e_2)=b_{5,2}, \\ [1mm] \omega(x_2,x_1)=-b_{4,5}.\end{array}$$
\item[7.] $\alpha_1=\alpha_2=0,$ $\beta_1=-\beta_2=-1:$
 $$\begin{array}{llllll} \omega(e_1,e_1)=b_{1,1}, &\omega(e_1,x_1)=-b_{4,1},& \omega(e_1,x_2)=b_{4,1},&\omega(e_2,x_1)=-b_{4,2},\\ [1mm] \omega(e_2,x_2)=-b_{5,2},&
 \omega(e_3,x_1)=2b_{1,1}, &\omega(e_3,x_2)=-b_{1,1}, & \omega(x_1,e_1)=b_{4,1},\\ [1mm] \omega(x_1,e_2)=b_{4,2},&
\omega(x_1,x_2)=b_{4,5},&
\omega(x_2,e_1)=-b_{4,1}, &
\omega(x_2,e_2)=b_{5,2}, \\ [1mm] \omega(x_2,e_3)=b_{1,1}, &\omega(x_2,x_1)=-b_{4,5}.\end{array}$$
\item[8.] $\alpha_1=-\alpha_2,$ $\beta_1=-\beta_2,$ $(\alpha_1,\beta_1)\neq(-1,0),(0,-1):$
$$\begin{array}{llllll} \omega(e_1,e_1)=b_{1,1}, & \omega(e_1,x_1)=-(1+\alpha_1)b_{5,1},& \omega(e_1,x_2)=-\beta_1b_{5,1},\\ [1mm] \omega(e_2,x_1)=-\alpha_1b_{5,2},&
\omega(e_2,x_2)=-(1+\beta_1)b_{5,2},&
\omega(e_3,x_1)=(2+\alpha_1)b_{1,1}, \\ [1mm]
\omega(e_3,x_2)=\beta_1b_{1,1}, & \omega(x_1,e_1)=(1+\alpha_1)b_{5,1},&
\omega(x_1,e_2)=\alpha_1b_{5,2},\\ [1mm]
\omega(x_1,e_3)=-\alpha_1 b_{1,1},&\omega(x_1,x_2)=b_{4,5},&
\omega(x_2,e_1)=\beta_1b_{5,1},\\ [1mm]
\omega(x_2,e_2)=(1+\beta_1)b_{5,2}, &
  \omega(x_2,e_3)=-\beta_1 b_{1,1},  & \omega(x_2,x_1)=-b_{4,5}.\end{array}$$

\end{itemize}
\end{prop}

\begin{proof}
The proof is similar to that of Propositions \ref{Z2forH} and \ref{Z2forL1}. \end{proof}

\begin{cor}\label{corL.3} We have the following:
\begin{itemize}
\item if $\alpha_1=\beta_1=\beta_2=0$, $\alpha_2=1,$  then  $\operatorname{dim} Z^2(L_3,l,r) =6,$
$\operatorname{dim} B^2(L_3,l,r) =5;$
\item if $\alpha_1=\beta_1=\beta_2=0$, $\alpha_2=3,$ then  $\operatorname{dim} Z^2(L_3,l,r) =6,$
$\operatorname{dim} B^2(L_3,l,r) =5;$
\item if $\alpha_1=\alpha_2=\beta_1=0,$ $\beta_2=1$, then  $\operatorname{dim} Z^2(L_3,l,r) =6,$
$\operatorname{dim} B^2(L_3,l,r) =5;$
\item if $\alpha_1=\alpha_2=\beta_1=0,$ $\beta_2=2$, then  $\operatorname{dim} Z^2(L_3,l,r) =6,$
$\operatorname{dim} B^2(L_3,l,r) =5;$
\item if $\alpha_1=\beta_1=0,$ $(\alpha_2,\beta_2)\neq (1,0),(3,0),(0,1), (1,1)$, then  $\operatorname{dim} Z^2(L_3,l,r) =5,$
$\operatorname{dim} B^2(L_3,l,r) =5;$
\item if $\alpha_1=-\alpha_2=-1,$ $\beta_1=\beta_2=0$, then  $\operatorname{dim} Z^2(L_3,l,r) =5,$
$\operatorname{dim} B^2(L_3,l,r) =3;$
\item if $\alpha_1=\alpha_2=0,$ $\beta_1=-\beta_2=-1$, then  $\operatorname{dim} Z^2(L_3,l,r) =5,$
$\operatorname{dim} B^2(L_3,l,r) =3;$
\item if $\alpha_1=-\alpha_2,$ $\beta_1=-\beta_2,$  $(\alpha_1,\beta_1)\neq(-1,0),(0,-1)$ then $\operatorname{dim} Z^2(L_3,l,r) =4,$
$\operatorname{dim} B^2(L_3,l,r) =4.$

 \end{itemize}
\end{cor}

\begin{thm} Let $\widehat{L}_3$ be an extension of the solvable Leibniz algebra $L_3$ by the abelian algebra $\mathfrak{h}=\langle e_4 \rangle.$ Then
   $\widehat{L}_3$ is isomorphic to one of the following non-isomorphic algebras:
$$\widehat{L}^1_3:\begin{array}{lllll} [e_1,e_1]=e_3,& [e_1,x_1]=e_1, &[x_1,e_1]=-e_1, & [x_2,e_1]=e_4,\\ [1mm]
  [e_3,x_1]=2e_3, & [e_2,x_2]=e_2,& [x_2,e_2]=-e_2,& [e_4,x_1]=e_4. \end{array} $$
$$\widehat{L}^2_3:\begin{array}{lllll} [e_1,e_1]=e_3,& [e_3,e_1]=e_4,& [e_1,x_1]=e_1, &[x_1,e_1]=-e_1, \\ [1mm]
  [e_3,x_1]=2e_3, & [e_2,x_2]=e_2,& [x_2,e_2]=-e_2,& [e_4,x_1]=3e_4. \end{array} $$
$$\widehat{L}^3_3:\begin{array}{lllll} [e_1,e_1]=e_3,& [e_1,x_1]=e_1, & [x_1,e_1]=-e_1,& [x_1,e_2]=e_4,  \\ [1mm] [e_3,x_1]=2e_3, & [e_2,x_2]=e_2, & [x_2,e_2]=-e_2, &[e_4,x_2]=e_4.\end{array} $$
 $$\widehat{L}^4_3:\begin{array}{lllll} [e_1,e_1]=e_3,& [e_2,e_2]=e_4,& [e_1,x_1]=e_1, & [x_1,e_1]=-e_1, \\ [1mm] [e_3,x_1]=2e_3, & [e_2,x_2]=e_2, & [x_2,e_2]=-e_2, &[e_4,x_2]=2e_4.\end{array} $$
 $$\widehat{L}^5_3:\begin{array}{lllll} [e_1,e_1]=e_3,& [e_1,x_1]=e_1+e_4, &[x_1,e_1]=-e_1-e_4, & [e_3,x_1]=2e_3,\\ [1mm]
  [e_2,x_2]=e_2,&[x_2,e_2]=-e_2, & [e_4,x_1]=e_4,& [x_1,e_4]=-e_4.\end{array} $$
 $$\widehat{L}^6_3(\delta):\begin{array}{lllll} [e_1,e_1]=e_3,& [e_1,x_1]=e_1+\delta e_4, &[x_1,e_1]=-e_1-\delta e_4,
  &[e_3,x_1]=2e_3, \\ [1mm]
  [e_1,x_2]=e_4,& [x_2,e_1]=-e_4,& [e_2,x_2]=e_2, & [x_2,e_2]=-e_2, \\ [1mm] [e_4,x_1]=e_4,& [x_1,e_4]=-e_4.\end{array} $$
$$\widehat{L}^7_3:\begin{array}{lllll} [e_1,e_1]=e_3,& [e_1,x_1]=e_1, &[x_1,e_1]=-e_1, & [e_3,x_1]=2e_3,\\ [1mm]
    [e_2,x_1]=e_4,&  [x_1,e_2]=-e_4,&[e_2,x_2]=e_2,&[x_2,e_2]=-e_2, \\ [1mm]
     [e_4,x_2]=e_4,& [x_2,e_4]=-e_4.\end{array} $$
 $$\widehat{L}^8_3(\delta):\begin{array}{lllll} [e_1,e_1]=e_3,& [e_1,x_1]=e_1, &[x_1,e_1]=-e_1,
  &[e_3,x_1]=2e_3, \\ [1mm]
  [e_2,x_1]=\delta e_4,& [x_1,e_2]=-\delta e_4,& [e_2,x_2]=e_2+e_4, & [x_2,e_2]=-e_2-e_4, \\ [1mm] [e_4,x_2]=e_4,& [x_2,e_4]=-e_4.\end{array} $$

 \end{thm}

\begin{proof} The proof is similar to that of Theorems \ref{thmH} and \ref{thmL1}. \end{proof}


\begin{thebibliography}{99}
\bibitem{Albeverio} Albeverio S., Omirov B.A., Rakhimov I.S., Varieties of nilpotent complex Leibniz algebras of dimensions less then five. Communications in Algebra. 2005,  33(5), 1575--1585.

\bibitem{Ayupov}  Ayupov Sh.A., Omirov B.A., On 3-dimensional Leibniz algebras. Uzbek Mathematical Journal, 1999, 1, 9--14.

\bibitem{Omirov} Ayupov Sh.A., Omirov B.A., On some classes of nilpotent Leibniz algebras. Siberian Mathematical Journal, 2001, 42(1), 15--24.

\bibitem{AyupovBook} Ayupov Sh.A., Omirov B.A., Rakhimov I.S., Leibniz Algebras: Structure and Classification, Taylor and Francis Group Publisher, ISBN 0367354810, 2019, 323 pp.

  \bibitem{Barnes} Barnes D.W., On Levi's theorem for Leibniz algebras. Bulletin of the Australian Mathematical Society, 2012, 86(2), 184--185.

 \bibitem{Bloch} Bloch A.M., On a generalization of the concept of Lie algebra. Doklady Akademii Nauk SSSR, 1965, 18(3), 471--473.


\bibitem {Canete} Ca\~{n}ete E.M., Khudoyberdiyev A.Kh., The classification of 4-dimensional Leibniz algebras. Linear Algebra and its Applications, 2013, 439(1), 273--288.


\bibitem{Camacho} Camacho L.M., Navarro R.M., Omirov B.A., Central extensions of some solvable Leibniz superalgebras. Linear Algebra and its Applications, 2023, 656(1), 63--91.

\bibitem{Casas1} Casas J.M., Faro E., Vieites A.M. Abelian extensions of Leibniz algebras. Communications in Algebra. 1999, 27(6), 2833--2846.

\bibitem{Casas} Casas J.M., Insua M.A., Ladra M., Ladra S., An algorithm for the classification of 3-dimensional complex Leibniz algebras. Linear Algebra and its Applications, 2012, 436(9), 3747--3756.

\bibitem{Karimjonov} Casas J.M, Ladra M., Omirov B.A., Karimjanov I.A.,
 Classification of solvable Leibniz algebras with null-filiform nilradical.  Linear and Multilinear Algebra, (2013), 61(6), 758--774.



\bibitem{Cuv} Cuvier C., Alg\'{e}bres de Leibnitz: d\'{e}finitions, propri\'{e}t\'{e}s.
Annales Scientifiques de l'\'{E}cole Normale Sup\'{e}rieure, 1994, 27(1), 1--45.



\bibitem{Kaygorodov} Kaygorodov I., Popov Y., Pozhidaev A., Volkov Y.,
Degenerations of Zinbiel and nilpotent Leibniz algebras. Linear and Multilinear Algebra, 2018, 66(4), 704--716.


\bibitem{Khudoyberdiyev} Khudoyberdiyev A. Kh., Rakhimov, I. S., Said Husain, Sh. K., On classification of $5$-dimensional solvable Leibniz algebras. Linear Algebra and its Applications, 2014, 457, 428--454.

\bibitem{Sheraliyeva1}  Khudoyberdiyev A. Kh., Sheraliyeva S., Extensions of solvable Lie algebras with naturally graded filiform nilradical. Journal of Algebra and its Applications, 2024, 23(10), 2450161.

\bibitem{Liu} Liu J., Sheng Y., Wang Q., On non-abelian extension of Leibniz algebras. Communications in Algebra, 2018, 46 (2), 574--587.

\bibitem{Loday} Loday J.-L., Une version non commutative des alg\`ebres de Lie: les alg\`ebres de Leibniz. L'Enseignement Math\'{e}matique, 1993, 39(3-4), 269--293.


\bibitem{Rakhimov}  Rakhimov I.S., Langari S.L., Langari M.B., On central extensions of null-filiform Leibniz algebras. Journal of Algebra, 2009, 3(6), 271--280.

\bibitem{Sheraliyeva} Sheraliyeva S., Extension of some solvable Leibniz algebra. Bulletin of the Institute of Mathematics, 2023, 6(4), 78--89.


\bibitem{Sund2} Sund T., On the structure of solvable Lie algebras. Mathematica Scandinavica, 1979, 44, 235--242.


\end{thebibliography}
\end{document}